\documentclass[12pt,reqno]{amsart}

\usepackage{epsfig}
\usepackage{amsmath, amsthm, amsfonts, amssymb, mathrsfs}
\usepackage{tikz-cd}
\usepackage{graphicx}

\numberwithin{equation}{section}
\DeclareMathOperator\Hom{Hom}

\newcommand{\PP}{{\mathbb P}}
\newcommand{\R}{{\mathbb R}}
\newcommand{\C}{{\mathbb C}}
\newcommand{\N}{{\mathbb N}}
\newcommand{\Q}{{\mathbb Q}}
\newcommand{\Z}{{\mathbb Z}}
\newcommand{\Zsstar}{\mathbb{Z}_s^\times}

\newcommand{\dpp}{\delta_{\PP}}

\newcommand{\rhot}{\rho_\theta}

\newcommand{\mnnorm}[3][M]{\left\|#2\right\|_{#1,#3}}
\newcommand{\nnorm}[1]{\left\|#1\right\|_N}
\newcommand{\Cnorm}[2]{\left\|#2\right\|_{C^{#1}}}

\newtheorem{theo}{{\sc \bf Theorem}}[section]

\newtheorem{lem}[theo]{{\sc \bf Lemma}}
\newtheorem{prop}[theo]{{\sc \bf Proposition}}

\newenvironment{defin}{\medskip\noindent{\bf Definition:\/} }{\medskip}

\begin{document}

\title{Noncommutative Geometry of Hensel-Steinitz Algebras}

\author[Hebert]{Shelley Hebert}
\address{Department of Mathematics and Statistics,
Mississippi State University,
175 President's Cir. Mississippi State, MS 39762, U.S.A.}
\email{sdh7@msstate.edu}

\author[Klimek]{Slawomir Klimek}
\address{Department of Mathematical Sciences,
Indiana University-Purdue University Indianapolis,
402 N. Blackford St., Indianapolis, IN 46202, U.S.A.}
\email{sklimek@math.iupui.edu}

\author[McBride]{Matt McBride}
\address{Department of Mathematics and Statistics,
Mississippi State University,
175 President's Cir., Mississippi State, MS 39762, U.S.A.}
\email{mmcbride@math.msstate.edu}

\author[Peoples]{J. Wilson Peoples}
\address{Department of Mathematics,
Pennsylvania State University,
107 McAllister Bld., University Park, State College, PA 16802, U.S.A.}
\email{jwp5828@psu.edu}

\date{\today}

\begin{abstract}
We discuss various aspects of noncommutative geometry of smooth subalgebras of Hensel-Steinitz algebras. In particular we study the structure of derivations and $K$-Theory of those smooth subalgebras.
\end{abstract}

\maketitle
\section{Introduction}
In this paper we study natural C$^*$-algebras which are related to multiplication maps of the ring of $s$-adic integers $\Z_s$. We name those algebras Hensel-Steinitz algebras. They are defined as crossed products by endomorphisms \eqref{alpha_def} of  $C(\Z_s)$ induced by $s$-adic multiplications on the space of $s$-adic integers. We provide a short appendix summarizing relevant properties of crossed products by monomorphisms with hereditary ranges. Hensel-Steinitz algebras are one of many interesting examples of C$^*$-algebras with connections to number theory, a subject of increasing interest, see for example \cite{CuLi}.

 Alternative C$^*$-algebras associated to endomorphisms of groups in general and to $s$-adic multiplication in particular have been considered before, see in particular \cite{CuntzVershik},  \cite{Hirshberg} and \cite{LarsenLi}  However, in those papers, the endomorphisms of  $C(\Z_s)$ are not the same as ours and they come from the endomorphisms $x\mapsto (x-x\text{ mod }s)/s$ of $\Z_s$. Consequently, the entire analysis and the properties of the corresponding C$^*$-algebras differ significantly.
 
In this paper we concentrate mostly on smooth subalgebras of Hensel-Steinitz algebras, which in noncommutative geometry capture a smooth structure on such ``quantum spaces". Smooth algebras play an important role in cyclic cohomology theory \cite{Me} and other noncommutative geometry applications \cite{Connes}. Further motivation for studying smooth subalgebras and general frameworks can be found in \cite{BC} and \cite{Bo}.

This work is a continuation of previous papers on the subject of noncommutative geometry of various examples of quantum spaces, in particular \cite{H}, \cite{KMP2}, \cite{KMP3}, \cite{KMP4}. A companion paper \cite{HKMP} discusses further generalizations to other multiplication maps in the ring of $p$-adic integers.

We verify that the smooth subalgebras studied in this paper are dense $*$-subalgebras closed under the holomorphic functional calculus and complete in their own stronger locally convex topology. We show also that they are closed under the smooth functional calculus of self-adjoint elements. 

Next, we look at derivations on smooth subalgebras of Hensel-Steinitz algebras. Derivations naturally arise in differential geometry as vector fields and play an analogous role in noncommutative geometry. It was observed in \cite{BEJ} that derivations on smooth subalgebras, modulo inner derivations, often have geometric significance. Using ``generalized'' Toeplitz operators we provide detailed structure of the space of continuous derivations on smooth subalgebras of Hensel-Steinitz algebras in Theorem \ref{der_theo}, a geometrically interesting and unexpected result.

Finally, we discuss $K$-Theory of Hensel-Steinitz algebras. $K$-Theory is a key invariant in the noncommutative topology of quantum spaces. We relate Hensel-Steinitz algebras with some familiar C$^*$-algebras and compute their $K$-Theory. By stability under the holomorphic functional calculus, smooth subalgebras have the same $K$-Theory as the Hensel-Steinitz algebras.

The paper is organized as follows.  In section 2 we define Hensel-Steinitz algebras and study their structure.  In section 3 we discuss smooth subalgebras of Hensel-Steinitz algebras and their properties.  In section 4 we classify continuous derivations on the smooth subalgebras from section 3.  In section 5 we discuss the $K$-Theory of Hensel-Steinitz algebras.  Finally, in section 6, the appendix, we review crossed products by monomorphisms with hereditary ranges.

\section{The Hensel-Steinitz Algebras}

For $s\in\N$ with $s\ge2$, the space of $s$-adic integers $\Z_s$ is defined as consisting of infinite sums:
\begin{equation*}
\Z_s= \left\{x=\sum_{j=0}^\infty x_j s^j: 0\le x_j\le s-1, j=0,1,2,\ldots\right\}\,.
\end{equation*}
We call the above expansion of $x$ its $s$-adic expansion.  Note that $\Z_s$ is isomorphic with the countable product $\prod \Z/s\Z$ and so, when the product is equipped with the Tychonoff topology, it is a Cantor set. It is also a metric space with the usual norm, $|\cdot|_s$, which is defined as follows: 
\begin{equation}\label{sadicnorm}
|0|_s=0 \textrm{ and, if }x\ne 0, |x|_s=s^{-n},
\end{equation}
where $x_n$ is the first nonzero term in the above $s$-adic expansion.

Moreover, $\Z_s$ is an abelian ring with unity with respect to addition and multiplication with a carry and we have:
\begin{equation*}
\Z_s\cong\lim_{\longleftarrow}\Z/s^n\Z\,,
\end{equation*}
for more details, see \cite{monothetic}.

The main objects of interest in this paper are the maps, $\alpha:C(\Z_s)\to C(\Z_s)$ defined by
\begin{equation}\label{alpha_def}
(\alpha f)(x) = \left\{
\begin{aligned}
&f\left(\frac{x}{s}\right) &&\textrm{ if }s|x \\
&0 &&\textrm{else}
\end{aligned}\right.
\end{equation}
and $\beta:C(\Z_s)\to C(\Z_s)$ given by
$$(\beta f)(x) = f(sx).$$
Notice that both maps are endomorphisms.  It is clear that $\alpha$ is an injection and we have the following relation:
\begin{equation*}
((\beta\alpha)f)(x) = \beta(\alpha f )(x) = (\alpha f)(sx) = f(x)
\end{equation*}
for any $f\in C(\Z_s)$.  Additionally, we have 
\begin{equation*}
\beta(1)=1 \textrm{ and } (\alpha\beta)f=\alpha(1)f.
\end{equation*}

Notice that
\begin{equation*}
\alpha (1)(x) = \left\{
\begin{aligned}
&1&&\textrm{ if }s|x \\
&0 &&\textrm{else}
\end{aligned}\right.
\end{equation*}
is the characteristic function of the subset of the ring $\Z_s$ consisting of those elements that are divisible by $s$. It follows from the formula for $\alpha$ that the range of $\alpha$ is equal to $\alpha(1)C(\Z_s)$. Thus $\alpha$ is a monomorphism with hereditary range and in particular $\beta$ satisfies equation \eqref{beta_def_ref} from the appendix. 

We define the Hensel-Steinitz algebra $HS(s)$ to be the following crossed product:
\begin{equation*}
HS(s) := C(\Z_s)\rtimes_\alpha\N\,,
\end{equation*}
see the appendix for details on crossed products by monomorphisms with hereditary range.

It is useful to realize $HS(s)$ as a concrete C$^*$-algebra. Let $H=\ell^2(\Z)$ and $\{E_l\}_{l\in\Z}$ be its canonical basis, then since $\Z$ is a dense subset of $\Z_s$, the mapping $C(\Z_s)\to B(H)$, $f\mapsto M_f$, given by 
$$M_fE_l = f(l)E_l$$
 is a faithful representation of $C(\Z_s)$.  
Let $V$ be the $s$-adic shift operator defined on $H$ via
\begin{equation*}
VE_l = E_{sl}\,.
\end{equation*}
A simple calculation verifies that
\begin{equation*}
V^*E_l = \left\{
\begin{aligned}
&E_{l/s} &&\textrm{ if } s|l \\
&0 &&\textrm{else.}
\end{aligned}\right.
\end{equation*}

Given $l\in\Z$, write $l=s^ml'$ and define the following diagonal operator in $H$ :
\begin{equation*}
\mathbb{P}E_{s^ml'} = mE_{s^ml'},\ \ \ \mathbb{P}E_{0} = 0.
\end{equation*}
It is clear that $\mathbb{P}$ is an unbounded self-adjoint operator on $H$.  We need this operator for the proof of the next theorem, however it is used substantially more in the upcoming sections.

\begin{theo}
There is an isomorphism between $C^*$-algebras:
\begin{equation*}
HS(s) = C(\Z_s)\rtimes_\alpha\N \cong C^*(V,M_f:f\in C(\Z_s))\,.
\end{equation*}
\end{theo}
\begin{proof}
To prove this we must first verify the relations between $V$ and $M_f$. From the definition of $V$, an immediate calculation shows that $V^*V=1$. 
Similarly we have that
\begin{equation*}
VM_fV^*E_l =\left\{
\begin{aligned}
&VM_f\left(\frac{1}{s}\right)E_{l/s} &&\textrm{ if } s|l \\
&0 &&\textrm{ else}
\end{aligned}\right.= \left\{
\begin{aligned}
&f\left(\frac{l}{s}\right)E_l &&\textrm{ if } s|l \\
&0 &&\textrm{ else}
\end{aligned}\right.=M_{\alpha f}E_l\,.
\end{equation*}
Thus $V$ and $M_f$ satisfy the relations of the crossed product of Stacey, see the appendix. Also, it follows from a direct calculation (and from general considerations in the appendix) that we have:
\begin{eqnarray}\label{strong_comm}
M_fV= VM_{\beta f}.
\end{eqnarray}

The algebra $C^*(V,M_f:f\in C(\Z_s))$ carries a natural circle action defined by maps $\rhot$ given by
\begin{equation}\label{rhotheta}
\rhot (a)=e^{2\pi i\theta \mathbb P} a e^{-2\pi i\theta\mathbb P}\,.
\end{equation}
A calculation on basis elements of $H$ shows that $\rhot(V)=e^{2\pi i\theta}V$ and that $\rhot(M_f)=M_f$.  It follows that $\rhot$ is a continuous 1-parameter family of automorphisms of $C^*(V,M_f)$ with $f\in C(\Z_s)$. Additionally, we infer from the relation $VV^*=\alpha(1)$ that the $\rhot$ invariant elements of $C^*(V,M_f)$ are precisely the multiplication operators $M_f$.

Next we consider an expectation in $C^*(V,M_f)$ onto the invariant subalgebra $C^*(M_f)\subseteq C^*(V,M_f)$. Define the map $E:C^*(V,M_f)\to C^*(V,M_f)$ by
\begin{equation*}
E(a)=\int_0^1 \rhot(a)\,d\theta \,.
\end{equation*}
Clearly, $E$ is an expectation in $C^*(V,M_f)$ onto $C^*(M_f)\cong C(\Z_s)$.

For any $a\in C^*(V,M_f)$ given by a finite sum
\begin{equation*}
a=\sum_{n\ge0}V^nM_{f_n} + \sum_{n<0}M_{f_n}(V^*)^{-n}\,,
\end{equation*}
we get
\begin{equation}\label{FourierCoeff}
M_{f_n} = \left\{
\begin{aligned}
&E((V^*)^{-n}a) &&\textrm{ if }n\ge0\\
&E(aV^n) &&\textrm{ if }n<0\,.
\end{aligned}\right.
\end{equation}
In, particular, we have that 
$$\|M_{f_0}\|\leq \|a\|,$$
which is the O'Donovan condition from the appendix. Finally, as mentioned above, the density of $\Z$ in $\Z_s$ implies that the representation is faithful.  Thus  we have $HS(s) \cong C^*(V,M_f)$.
\end{proof}

Below we consider the Hensel-Steinitz algebra $HS(s)$ as the concrete C$^*$-algebra from the above proposition. To understand the structure of $HS(s)$ we first look at an ideal, denoted by $I_s$, which can be defined in the following way. Consider the map $C(\Z_s) \to B(\ell^2(\Z))$ defined by $f \mapsto m_f$, where $m_f$ is given by  
$$m_f E_l = f(0) E_l.$$
Additionally, let $v \in B(\ell^2(\Z))$ denote the standard bilateral shift 
$$vE_l = E_{l+1}.$$ 
It is easy to check that sending $V \mapsto v$ and $M_f \mapsto m_f$ gives rise to a representation of $HS(s)$ onto $C^*(v,m_f) \cong C(\R/\Z)$. Denote this surjective representation by 
$$\pi_0: HS(s) \to C(\R/\Z).$$ 
We define $I_s$ to be the ideal which is the kernel of this representation:
$$
I_s : = \textup{Ker } \pi_0. 
$$
The ideal $I_s \subseteq HS(s)$ can also be described as the subalgebra generated by the operators $V^m M_{f}$, $m\geq 0$
and $M_{f}(V^*)^{-m}$, $m<0$ such that $f\in C(\Z_s)$ with $f(0)=0$.

Let $\Z_s^\times$ be the unit ball in $\Z_s$.  Notice that if $x\in \Z_s$ has the $s$-adic expansion:
\begin{equation*}
x=x_0 + x_1s + x_2s^2+\cdots
\end{equation*}
then, $x\in\Z_s^\times$ if and only if $x_0\neq0$. $\Z_s^\times$ is both a closed and open subset of $\Z_s$. Let $\mathcal{K}$ be the ideal of compact operators in $B(H)$.

\begin{theo}
\label{ideal isomorphism}
There is the following isomorphism of C$^*$-algebras:
\begin{equation*}
I_s\cong C(\Z_s^\times)\otimes\mathcal{K}\,.
\end{equation*}
\end{theo}
\begin{proof}
As noticed above, we have:
\begin{equation*}
I_s \cong C^*(V^mM_{f} : f\in C(\Z_s), f(0)=0).
\end{equation*}
Suppose $f\in C(\Z_s)$ with $f(0)=0$.  For $x\neq 0$ we can uniquely write $f(x)=f(s^m\tilde{x})$ where $x=s^m\tilde{x}$ and $|\tilde{x}|_s=1$, so $\tilde{x}\in \Z_s^\times$.  Consequently, we have an isomorphism of C$^*$-algebras  
\begin{equation*}
\{f\in C(\Z_s):f(0)=0\} \cong \left\{F\in C(\Z_{\ge0}\times \Zsstar):\lim_{m\to\infty} F(m,\cdot)=0\right\} := C_0(\Z_{\ge0}\times\Zsstar)
\end{equation*}
where $f \mapsto F$ is given by $f(s^m\tilde{x})=F(m,\tilde{x})$.
Additionally, we also have the following identification:
\begin{equation*}
C_0(\Z_{\ge0}\times\Zsstar) \cong c_0(\Z_{\ge0})\otimes C(\Z_s^\times)\,.
\end{equation*}

%

Given nonzero $l$, decompose $l=s^ml'$, with $l'\in\Z\subseteq\Z_s$, $|l'|_s=1$. For $\lambda \in C(\Zsstar)$, define the multiplication operator $M_\lambda :H\to H$ by the following
\begin{equation}\label{Mphi}
M_\lambda E_l=\left\{
\begin{aligned}
&\lambda(l')E_l &&\textrm{if }l\ne 0 \\
&0 &&\textrm{if }l=0.
\end{aligned}\right.
\end{equation}
Similarly, given a sequence $\chi=\{\chi(l)\} \in c_0(\Z_{\ge0})$ and the decomposition above, define $M_\chi:H\to H$ by
\begin{equation*}
M_\chi E_l = \left\{
\begin{aligned}
&\chi(m)E_l &&\textrm{if }l\ne 0 \\
&0 &&\textrm{if }l=0.
\end{aligned}\right.
\end{equation*}
Then, we have the following isomorphism of $C^*$-algebras:
\begin{equation*}
C^*(M_f : f\in C(\Z_s), f(0)=0) \cong C^*(M_\lambda M_\chi : \lambda\in C(\Zsstar), \chi\in c_0(\Z_{\ge0}))\cong C(\Zsstar)\otimes c_0(\Z_{\ge0})\,.
\end{equation*}

Notice we have the following commutation relations: 
\begin{equation} 
\label{MxMphirel}
M_\lambda M_\chi = M_\chi M_\lambda \text{ and } M_\lambda V = VM_\lambda.
\end{equation}  
Using the relations in equation \eqref{MxMphirel}  the following two C$^*$-algebras are isomorphic:
\begin{equation*}
\begin{aligned}
&C^*(V^nM_{f}:f\in C(\Z_s), f(0)=0, n\in \Z_{\ge0}) \\  &C^*(M_{\lambda}:\lambda\in C(\Zsstar))\otimes C^*(V^nM_\chi: n\in \Z_{\ge0}, \chi\in c_0(\Z_{\ge0})).
\end{aligned}
\end{equation*}

We claim the following isomorphism:
\begin{equation*}
C^*(V^nM_\chi: n \in \Z_{\ge0}, \chi\in c_0(\Z_{\ge0})) \cong \mathcal{K}\,.
\end{equation*}
For $j\ge 0$ let $\chi_j\in c_0(\Z_{\ge0})$ be a sequence given by $\chi_j=\{\delta_{ij}\}_{i\ge 0}$ and define the operators $P_{ij}$ by
\begin{equation}\label{Pij}
P_{ij}=\left\{
\begin{aligned}
&V^{i-j} M_{\chi_j} &&\textrm{if }i\ge j \\
&M_{\chi_i}(V^*)^{j-i} &&\textrm{if }j\ge i\,.
\end{aligned}\right.
\end{equation}
Clearly, $\{P_{ij}\}$ form the system of units of $\mathcal{K}$, thus we have
\begin{equation*}
C^*(P_{ij}) \cong \mathcal{K}\,.
\end{equation*}
Finally, $C^*(P_{ij})$ and $C^*(V^nM_\chi: n \in \Z_{\ge0}, \chi\in c_0(\Z_{\ge0}))$ clearly coincide and thus our claim holds, finishing the proof.
%
\end{proof}
This theorem leads to the short exact sequence describing the structure of $HS(s)$:
$$
0 \to C(\Z_s^\times)\otimes\mathcal{K} \to HS(s) \to C(\R/\Z) \to 0,
$$
which is instrumental in further analysis of the Hensel-Steinitz algebra. The quotient map in the above short exact sequence will be denoted by 
$$q:HS(s)\to C(\R/\Z).$$

For completeness, we briefly describe two more structural results regarding $HS(s)$. The first observation realizes $HS(s)$ as a full corner of a regular crossed product by an automorphism C$^*$-algebra, see also \cite{St} and \cite{Laca} for general constructions of this type. The second statement identifies $HS(s)$ with a subalgebra of the tensor product $C(\Z_s^\times)\otimes\mathcal{T}$, where $\mathcal{T}$ is the Toeplitz algebra. This identification motivates the concept of ``generalized'' Toeplitz operators introduced in the next section.

Consider the space
\begin{equation*}
\Q_s:= \left\{x=\sum_{j=N}^\infty x_j s^j: N\in\Z, 0\le x_j\le s-1, x_N\ne 0\right\}\,.
\end{equation*}
It is an abelian ring with unity with respect to the usual addition and multiplication with a carry to the right.
Equipped with the $s$-adic norm given by equation \eqref{sadicnorm}, $\Q_s$ is a locally compact space containing $\Z_s$. The key property of the ring $\Q_s$ is that $s$ is now invertible. The analog of the endomorphism $\alpha$ is an automorphism $\tilde\alpha:C_0(\Q_s)\to C_0(\Q_s)$ given by:
\begin{equation*}
(\tilde\alpha f)(x)=f\left(\frac{x}{s}\right).
\end{equation*}

Consider the crossed product $C_0(\Q_s)\rtimes_{\tilde\alpha}\Z$
of $C_0(\Q_s)$ by the automorphism $\tilde\alpha$. Notice that the characteristic function $1_{\Z_s}\in C_0(\Q_s)\subseteq C_0(\Q_s)\rtimes_{\tilde\alpha}\Z$ is a projection in $C_0(\Q_s)\rtimes_{\tilde\alpha}\Z$ and we can identify $HS(s)$ with the corner:
\begin{equation}\label{corneriso}
HS(s)\cong 1_{\Z_s}\cdot \left(C_0(\Q_s)\rtimes_{\tilde\alpha}\Z\right)\cdot 1_{\Z_s}.
\end{equation}
Indeed, we have inclusion $C(\Z_s)\ni f\mapsto f\cdot 1_{\Z_s}\in 1_{\Z_s}\cdot \left(C_0(\Q_s)\rtimes_{\tilde\alpha}\Z\right)\cdot 1_{\Z_s}$. Additionally, if $U$ is the unitary generating $\Z$ in the crossed product $C_0(\Q_s)\rtimes_{\tilde\alpha}\Z$ then $V:=1_{\Z_s}\cdot U\cdot 1_{\Z_s}$ is an isometry in the corner $1_{\Z_s}\cdot \left(C_0(\Q_s)\rtimes_{\tilde\alpha}\Z\right)\cdot 1_{\Z_s}$. It is easy to verify that the defining relations for the crossed product $C(\Z_s)\rtimes_\alpha\N$ are satisfied and indeed we have an isomorphism in equation \eqref{corneriso}.

Moreover, since under the map $x\mapsto x/s$ in $\Q_s$, the smallest invariant set containing $\Z_s$ is $\Q_s$ itself, it is easy to deduce that the smallest closed two-sided ideal containing the corner $1_{\Z_s}\cdot \left(C_0(\Q_s)\rtimes_{\tilde\alpha}\Z\right)\cdot 1_{\Z_s}$ is the full algebra $C_0(\Q_s)\rtimes_{\tilde\alpha}\Z$ and thus the corner is full.

Our second observation relates Hensel-Steinitz algebra with the Toeplitz algebra. Notice that for the tensor product $C(\Z_s^\times)\otimes\mathcal{T}$, we have 
the following short exact sequence:
$$
0 \to C(\Z_s^\times)\otimes\mathcal{K}\to C(\Z_s^\times)\otimes\mathcal{T} \to C(\Z_s^\times\times\R/\Z) \to 0.
$$
Observe that the left-most term in the above sequence is isomorphic to $I_s$ and that $C(\R/\Z)$ is the subalgebra of $C(\Z_s^\times\times\R/\Z)$ of functions independent of the first variable.
Denote the quotient map in the above short exact sequence by 
$$\sigma:C(\Z_s^\times)\otimes\mathcal{T} \to C(\Z_s^\times\times\R/\Z).$$
Then, using formulas \eqref{MxMphirel}, it is not difficult to see that we have:
\begin{equation*}
HS(s)\cong \sigma^{-1}(C(\R/\Z))=\{a\in C(\Z_s^\times)\otimes\mathcal{T}: \sigma(a)\subseteq C(\R/\Z)\}.
\end{equation*}
Notice in particular that $I_s\cong C(\Z_s^\times)\otimes\mathcal{K}=\sigma^{-1}(0)\subseteq \sigma^{-1}(C(\R/\Z))\cong HS(s)$.

\section{Smooth Subalgebras}
In this section we introduce natural dense $*$-subalgebras of Hensel-Steinitz algebras $HS(s)$ and we verify that they satisfy the properties of smooth subalgebras of \cite{BC}. Namely, we show that they are closed under the holomorphic functional calculus, they are complete in their own stronger locally convex topology and that they are closed under the smooth functional calculus of self-adjoint elements.  It should be noted that in our case, the definition of smooth subalgebras doesn't fit the entire framework of \cite{BC}, nor \cite{Bo}, and thus requires verifications of all the properties.  Studying these smooth subalgebras is a necessity to classifying derivations in the next section.

\subsection{Smooth Ideal $I_s^\infty$.}
Similarly to the approach to smooth subalgebras in \cite{KMP2} and \cite{KMP3}, we first define a smooth subalgebra of the ideal $I_s$. This is done by using Fourier coefficients of elements of $HS(s)$ with respect to the circle action given by equation \eqref{rhotheta}.

For $a\in HS(s)$ and $n\in\Z$ let $f_n\in C(\Z_s)$ be defined by the equation \eqref{FourierCoeff}.
If $a\in I_s$ we again use the decomposition of nonzero elements of $\Z_s$ as $s^m x$, $m\in \Z_{\ge0}$, $|x|_s=1$ and we put 
$$F_n(m,x)=f_n(s^m x).$$   
Functions $F_n\in C_0(\Z_{\ge0}\times \Zsstar)$ are called the Fourier coefficients of $a\in I_s$. By the usual Fourier series theory \cite{K}, the coefficients $F_n$ fully determine $a$. 

The smooth subalgebra $I_s^\infty$ of $I_s$ is defined as
\begin{equation*}
I_s^\infty := \left\{a\in I_s: \{\underset{x}{\textrm{sup}}|F_n(m,x)|\}_{m,n}\textrm{ is RD}\right\}\,.
\end{equation*}
Here RD means the coefficients are rapid decay.  In particular, the Fourier coefficients of $a\in I_s^\infty$
are in the space
\begin{equation*}
C_0^\infty(\Z_{\ge0}\times\Zsstar):=\left\{F\in C_0(\Z_{\ge0}\times\Zsstar):  \{\underset{x}{\textrm{sup }}|F(m,x)|\}_m\textrm{ is RD}\right\}\,.
\end{equation*}
Appealing again to general Fourier series theory we have that if $a\in I_s^{\infty}$ then $a$ has a norm convergent series expansion:
\begin{equation*}
a=\sum_{n\ge0}V^nM_{F_n} + \sum_{n<0}M_{F_n}(V^*)^{-n}    
\end{equation*}
where $F_n\in C_0^\infty(\Z_{\ge0}\times \Zsstar)$.

In what follows, we rephrase the definition of $I_s^{\infty}$ in terms of a family of norms and endow $I_s^{\infty}$ with a locally convex topology.  Given  $l=s^ml'$ as above, recall the following diagonal operator in $H=\ell^2(\Z)$ :
\begin{equation*}
\mathbb{P}E_{s^ml'} = mE_{s^ml'},\ \ \ \mathbb{P}E_{0} = 0.
\end{equation*}  


Given an $a\in HS(s)$, define the following commutator operator (a derivation) $\delta_\mathbb{P}$ on $HS(s)$ by:
\begin{equation*}
\delta_\mathbb{P}(a)=[\mathbb{P},a]\,.
\end{equation*}
Clearly this operator is only densely defined on $HS(s)$.

For non-negative integer $N$, define $\|\cdot\|_N : B(H) \to \R \cup \{\infty\}$ by $\|a\|_N = \|a(1+\PP)^N\|$ where $\|\cdot\|$ is the norm on $B(H)$.  As $\|\cdot\|$ is a norm on $B(H)$, so too is $\|\cdot\|_N$.  Note that $\|\cdot\|_0 = \|\cdot\|$.
\begin{prop}\label{nprop}
Let $a$ and $b \in B(H)$.  The following hold: 
\begin{enumerate}
\item $\|a\|_N$ is increasing in $N$.
\item $\|ab\|_N \leq \|a\| \|b\|_N \leq \|a\|_N \|b\|_N$.
\item $\|a^*\|_N \leq \sum_{j=0}^N \binom{N }{j} \|\delta_{\PP}^j(a)\|_N$.
\end{enumerate}
\end{prop}

\begin{proof}
Let $\mathscr{D} \subseteq H$ be a dense subspace in the Hilbert space $H$ given by
\begin{equation*}
\mathscr{D} = \left\{x\in H : x=\sum x_{m,l} E_{s^ml} \textrm{ such that } m \in \Z_{\ge0}, l\in \Z, s \nmid l, \{x_{m,l}\}\textrm{ is } \mathrm{\ RD}\right\}\,.
\end{equation*}
Note that $\mathscr{D}$ is in the domain of $(1+\PP)^N$ for every $N$.  Moreover
\begin{equation*}
\|a\|_N = \|a(1+\PP)^N\| = \underset{x\in \mathscr{D}, \|x\|\leq 1}{\textrm{sup }}\|a(1+\PP)^N x\|\,.
\end{equation*}
Since $\PP$ is a positive operator, notice that $\|(1+\PP)^N x\|\ge\|x\|$.  Therefore it follows that we have $\|(1+\PP)^{N+1}x\| \geq \|(1+\PP)^N x\|$.  By continuity of the norm and density of $\mathscr{D}$, we have $\|a\|_{N+1} \geq \|a\|_N$.

Item 2 follows directly from the submultiplicativity of $\|\cdot\|$ and Item 1.  The last Item follows from an induction argument similar to Proposition 2.2 in \cite{KMP3}.
\end{proof}

We define the (possibly infinite valued)  $M,N$-norm on $B(H)$  by
\begin{equation*}
\|a\|_{M,N} = \sum_{j=0}^M \binom{M}{j} \|\delta_{\PP}^j(a)\|_N\,.
\end{equation*}
It is clear from the definition that those expressions satisfy the norm properties for each $M,N\in \Z_{\ge0}$.  We can immediately see that $\mnnorm[0]{\cdot}{N}=\nnorm{\cdot}$. 

Our previous papers \cite{KMP2,KMP3} discussed similarly defined norms. Our next result describes properties of these norms, analogous to those demonstrated in those previous papers.  While the proofs are almost identical to the proofs of those other papers, we comment on some of the details below.

\begin{prop}\label{mnprop}
Let $a$ and $b$ be in $B(H)$. The following hold:
\begin{enumerate}
\item $a\in I_s^\infty$ if and only if $\|a\|_{M,N}<\infty$ for all nonnegative integers $M$ and $N$.
\item $\|a\|_{M+1,N} = \|a\|_{M,N} + \|\delta_\mathbb{P}(a)\|_{M,N}$.
\item $\|a\|_{M,N} \le \|a\|_{M,N+1}$.
\item $\|ab\|_{M,N} \le \|a\|_{M,0}\|b\|_{M,N} \le \|a\|_{M,N}\|b\|_{M,N}$.
\item $\|\delta_\mathbb{P}(a)\|_{M,N}\le \|a\|_{M+1,N}$.
\item $\|a^*\|_{M,N}\le \|a\|_{M+N,N}$.
\item $I_s^\infty$ is a complete topological vector space.
\end{enumerate}
\end{prop}

\begin{proof}
Suppose $a\in I_s^{\infty}$. Then $a$ has a norm convergent series expansion.  We write $a$ as
\begin{equation*}
a=\sum_{n\ge0}V^nM_{F_n} + \sum_{n<0}M_{F_n}(V^*)^{-n}    
\end{equation*}
where $F_n\in C_0(\Z_{\ge0}\times \Zsstar)$.  
Applying $\dpp$ to $V$ and $M_F$ for some $F\in C(\Z_s)$, and noting that $M_F$ is a diagonal operator, we have
\begin{equation*}
\dpp(V)=V,\quad \dpp(V^*)=-V^*, \quad\text{and}\quad  \dpp(M_F)=0 \,.    
\end{equation*}
It follows that
\begin{equation}
\label{steve}
\dpp^j(a)=\sum_{n\geq 0} n^j V^n M_{F_n} + \sum_{n<0} n^j M_{F_n} (V^*)^{-n} \,. 
\end{equation}  
Notice on a basis element we have that
\begin{equation*}
M_{F_n}(1+\PP)^NE_{s^mj} = F_n(m,j)(1+m)^NE_{s^mj}\,.
\end{equation*}
Since $F_n(m,j)$ is RD in both $m$ and $n$ and bounded in $j$, it follows that $\|a\|_{M,N} < \infty$.

Now assume $\mnnorm{a}{N} < \infty$.  We want to prove that $\{\sup_x |F_n(m,x)|\}_{m,n}$ is RD in $m$ and $n$.  If $n\ge 0$, from the expectation formula, we have 
\begin{equation} \label{stephen}
M_{F_n} = E((V^*)^n a)\,.
\end{equation}
It follows from the definition of the expectation that
\begin{equation*}
E(a(1+\PP)^N) = E(a)(1+\PP)^N.
\end{equation*}
Consequently,
$$M_{F_n}(1+\PP)^N = E((V^*)^na(1+\PP)^N),$$ 
and so, using $\|E(a)\| \leq \|a\|$, we have
\begin{equation*}
\|M_{F_n}(1+\PP)^N\| \leq \|V^*\|^n\nnorm{a} = \nnorm{a} < \infty\,.
\end{equation*}

Since $M_{F_n}(1+\PP)^N$ is a diagonal operator we can compute its norm and, using the above inequality, we get
\begin{equation*}
\sup_m \sup_{x\in \Z_s^\times} (1+m)^N|F_n(m,x)| = \sup_{m,l}(1+m)^N|F_n(m,l)| =\|M_{F_n}(1+\PP)^N\| < \infty\,.
\end{equation*}
A similar argument works for $n<0$.  Hence, we have the rapid decay in $m$.  

To get the rapid decay in both $m$ and $n$, first consider $n\ge 0$. 
It now follows from the expectation formula \eqref{stephen} and equation \eqref{steve} that 
\begin{equation*}
\sup_{n\ge 0} \|n^jM_{F_n}(1+\PP)^N\| = \sup_{n\ge 0} \sup_m \sup_x n^j(1+m)^N|F_n(m,x)| < \infty
\end{equation*}
for all $j, N$. There is a similar argument for $n<0$.  Therefore, $a\in I_s^\infty$.

Items (2)-(6) follow from either an induction argument, Proposition \ref{nprop} or a combination of them.  The last item follows from standard completeness arguments.
\end{proof}

Recall from Proposition 2.1 in \cite{KMP3}, that the space of smooth compact operators $\mathcal{K}^{\infty}$  can be described as
\begin{equation*}
\mathcal{K}^{\infty} = \left\{ \sum_{n\ge0} V^n x_n + \sum_{n< 0} x_n (V^*)^{-n} : x_n\in c_0(\Z_{\geq 0}), \{x_n(m)\}_{m,n} \text{ is RD} \right\}\,.
\end{equation*} Notice that it is nuclear Fr\'echet space.  Thus we have the following isomorphism of algebras:
\begin{equation*}
I_s^{\infty} \cong C(\Zsstar) \otimes \mathcal{K}^{\infty}.
\end{equation*}

\subsection{Holomorphic Stability of $I_s^\infty$}
Next we turn to stability properties of the smooth subalgebra $I_s^\infty$.

\begin{prop}\label{IpluscInv}
If $c\in I_s^\infty$ and $I+c$ is invertible in $HS(s)$, then
\begin{equation*}
(I+c)^{-1}-I = -(I+c)^{-1}c\in I_s^\infty\,.
\end{equation*}
\end{prop}

\begin{proof}
Since $I+c$ is invertible in $HS(s)$ we obtain
\begin{equation*}
(I+c)^{-1}-I = -(I+c)^{-1}c.
\end{equation*}
Next, consider the following calculation:
\begin{equation*}
\begin{aligned}
\dpp((I+c)^{-1}c) &= \dpp((I+c)^{-1})c + (I+c)^{-1}\dpp(c) \\
&=-(I+c)^{-1}\dpp(c)(I+c)^{-1}c + (I+c)^{-1}\dpp(c)
\end{aligned}
\end{equation*}
Notice that both terms in that equation are of the form $ab$ where $a$ is a bounded operator and $b$ is in $I_s^{\infty}$.  In fact, by induction we have for any $j$
\begin{equation*}
\dpp^j((I+c)^{-1}c) = \sum_i a_i b_i,\text{ a finite sum,}
\end{equation*}
where $a_i$ are bounded operators and $b_i\in I_s^\infty$.
Therefore, from Proposition \ref{nprop}, part $(2)$, we have that for every $N$
\begin{equation*}
\|\dpp^j((I+c)^{-1}c)\|_{0,N} < \infty
\end{equation*}
and so for every $M$ and $N$ we get
\begin{equation*}
\|(I+c)^{-1}-I\|_{M,N}<\infty
\end{equation*}
as desired.
\end{proof}

\begin{theo}
$I_s^\infty$ is closed under the holomorphic functional calculus: for any $a\in I_s^\infty$ and $f$ holomorphic on an open set containing the spectrum of $a$ with $f(0)=0$, we have $f(a)\in I_s^\infty$.
\end{theo}
\begin{proof}
Let $a\in I_s^\infty$, $f$ a holomorphic function on open set that contains the spectrum of $c$. If $\gamma$ is a curve around the spectrum of $a$ contained in the domain of $f$, define $f(a)$ via the Cauchy Integral Formula:
\begin{equation*}
f(a)=\frac{1}{2\pi i} \int_\gamma f(z)(z-a)^{-1}\,dz\,.
\end{equation*}
Completeness of $I_s^\infty$ and the Lebesgue Dominated Convergence Theorem yield convergence of the integral; $f(0)=0$ and Proposition \ref{IpluscInv} implies that $f(a)$ is in $I_s^\infty$, see \cite{Bo} for a more general discussion.
\end{proof}

\subsection{Smooth Stability of $I_s^\infty$.} We now consider stability of $I_s^\infty$ under smooth functional calculus of self-adjoint elements. We again are following \cite{KMP3}.

The next proposition will require an additional derivation, related to $\dpp$.  Define the continuous derivation $\partial_j$ on $I_s^\infty$ by
\begin{equation}\label{partialj}
\partial_j(a)=[(I+\PP)^j,a].
\end{equation}

It is clear that $\partial_1 = \dpp$.  A straightforward induction argument establishes the following relations between  $\partial_j$ and powers of $\dpp$:
\begin{equation}\label{partialjdpp}
\partial_j(a)=\sum_{k=1}^j \binom{j}{k}\dpp^k(a)(I+\PP)^{j-k}
\end{equation}
and 
\begin{equation}\label{dpppartialj}
\dpp^j(a)=\sum_{k=1}^j (-1)^{j-k} \binom{j}{k} \partial_k(a)(I+\PP)^{j-k}
\end{equation}
using the convention that $(I+\PP)^0=I$.  These formulas lead to the following observation.

\begin{prop}\label{partialdppequiv}
The norms $\|a\|_{M,N}$, $M, N=0,1,2\ldots$ are equivalent to the seminorms $\|\partial_j(a)\|_{0,N}$, $j,N=0,1,2,\ldots$.
\end{prop}
\begin{proof}
This follows directly from the definition of the $\|\cdot\|_{M,N}$ norm and equations \eqref{partialjdpp} and \eqref{dpppartialj}.
\end{proof}

\begin{prop}\label{expcminusI}
If $a\in I_s^\infty$ then $e^{ia}-I \in I_s^\infty$ and we have the following estimates:
 \begin{equation*}
\|e^{ia}-I\|_{0,N} \leq \|a\|_{0,N}.
\end{equation*}
and for $j\geq1$,
\begin{equation*}
\|\partial_j(e^{ia}-I)\|_{0,N}\leq \|\partial_j(a)\|_{0,N}+  \|\partial_j(a)\| \|a\|_{0,N}\,.
\end{equation*}
\end{prop}

\begin{proof} 
First we estimate the $0,N$-norm of $e^{ia}-I$. To do this we write:
\begin{equation*}
e^{ia}-I= \int_0^1\frac{d}{dt}\left(e^{ita}\right)\,dt=i \int_0^1e^{ita}a\,dt\,.
\end{equation*}
It follows from Proposition \ref{nprop} that we have:
\begin{equation}\label{eic_estimate}
\|e^{ia}-I\|_{0,N}\leq \int_0^1\|e^{ita}\|\|a\|_{0,N}\,dt =\|a\|_{0,N}\,.
\end{equation}

For $j\geq 1$ we write:
\begin{equation}\label{duhamel}
\begin{aligned}
\partial_j(e^{ia})&=(I+\PP)^je^{ia}-e^{ia}(I+\PP)^j=\int_0^1\frac{d}{dt}\left(e^{i(1-t)a}(I+\PP)^je^{ita}\right)\,dt=\\
&=
i\int_0^1e^{i(1-t)a}\partial_j(a)e^{ita}\,dt=i\int_0^1e^{i(1-t)a}\partial_j(a)\,dt+i\int_0^1e^{i(1-t)a}\partial_j(a)(e^{ita}-I)\,dt\,.
\end{aligned}
\end{equation}
Consequently, we estimate, using equation \eqref{eic_estimate}:
\begin{equation*}
\begin{aligned}
\|\partial_j(e^{ia}-I)\|_{0,N}&\leq \int_0^1\|e^{i(1-t)a}\partial_j(a)\|_{0,N}\,dt+\int_0^1\|e^{i(1-t)a}\partial_j(a)(e^{ita}-I)\|_{0,N}\,dt\leq\\
&\leq \|\partial_j(a)\|_{0,N}+  \|\partial_j(a)\| \|a\|_{0,N}\,,
\end{aligned}
\end{equation*}
by Proposition \ref{nprop}.


\end{proof}

\begin{theo}
$I_s^\infty$ is closed under the smooth functional calculus of self-adjoint elements: for any self-adjoint $c\in I_s^\infty$ and smooth $f$ on a neighborhood of the spectrum of $c$ with $f(0)=0$, we have $f(c)\in I_s^\infty$.
\end{theo}

\begin{proof} The proof below is inspired by ideas in \cite{BC}. Pick a number $L$ bigger than $2\|c\|$ and extend $f$ to the whole of $\R$ in such a way that it is not only still smooth, but also it is $L$-periodic: 
$$f(\theta+L)=f(\theta)$$ 
for every $\theta$. Then $f(\theta)$ admits a Fourier series representation:
\begin{equation*}
f(\theta)=\sum_{n\in\Z}f_ne^{2\pi in\theta/L}
\end{equation*}
with rapid decay coefficients $\{f_n\}$. Taking into account the assumption $f(0)=0$, we have:
\begin{equation*}
f(c)=\sum_{n\in\Z}f_n\left(e^{2\pi inc/L}-I\right).
\end{equation*}
Using the previous proposition we see that the seminorms
\begin{equation}
\|\partial_j(e^{2\pi inc}-I)\|_{0,N}
\end{equation}
grow at most quadratically in $n$, so that $\|\partial_j(f(c))\|_{M,N}<\infty$, by RD of $\{f_n\}$, implying that $f(c)$ is in $I_s^\infty$.
\end{proof}


\subsection{Smooth Hensel-Steinitz Algebra}
For $\phi\in C^\infty(\R/\Z)$, a Toeplitz operator, $T(\phi)$, acting on $H$ is given by
\begin{equation*}
T(\phi) = \sum_{m\geq 0}\phi_mV^m + \sum_{m<0}\phi_m(V^*)^{-m}
\end{equation*}
where $\{\phi_m\}$ are the $m^{\textrm{th}}$ Fourier coefficients of $\phi$.  Since $\phi\in C^\infty(\R/\Z)$, it follows that the sequence $\{\phi_m\}$ is RD and the above series is norm convergent.

We define the smooth Hensel-Steinitz algebra, $HS^\infty(s)$ by:
\begin{equation*}
HS^\infty(s) = \{a\in HS(s): a=T(\phi)+c \text{ with } \phi\in C^\infty(\R/\Z), c\in I_s^\infty\}\,.
\end{equation*}

We use the following function norms on $C^\infty(\R/\Z)$:
\begin{equation*}
\Cnorm{k}{\phi}=\sum_{j=0}^k \binom{k}{j}\left\|\left(\dfrac{1}{2\pi i} \frac{d}{d\theta}\right)^j \phi\right\|_\infty
\end{equation*}
It is straightforward to verify that they are submultiplicative and that 
\begin{equation}\label{cnorm}
\Cnorm{k+1}{\phi}=\Cnorm{k}{\phi}+\Cnorm{k}{\dfrac{1}{2\pi i}\frac{d}{d\theta} \phi}.
\end{equation} 

The basic properties of $HS^\infty(s)$ are established in the following statement.

\begin{prop}\label{IVIdeal}
If $\phi\in C^\infty(\R/\Z)$ and $c\in I_s^\infty$, then both $T(\phi)c$ and $cT(\phi)$ are in $I_s^\infty$.  If $\phi$ and $\psi$ are in $C^\infty(\R/\Z)$, then
\begin{equation*}
T(\phi)T(\psi)- T(\phi\psi)\in I_s^\infty\,.
\end{equation*}
Consequently, $HS^\infty(s)$ is a dense $*$-subalgebra of $HS(s)$ with $a^* = T({\overline{f}}) + c^*$ for a given $a\in HS^\infty(s)$.  Moreover, $I_s^\infty$ is a two-sided ideal of $HS^\infty(s)$ with
\begin{equation*}HS^\infty(s)/I_s^\infty\cong C^\infty(\R/\Z)\,,
\end{equation*}
and the quotient map $q:HS^\infty(s)\to C^\infty(\R/\Z)$ acts as 
$${q}(T(\phi)+c) = \phi$$ 
with $\phi\in C^\infty(\R/\Z)$.
\end{prop}

\begin{proof}
From Proposition \ref{nprop} we have that $\nnorm{T(\phi)c} \leq \|T(\phi)\|\nnorm{c}$.  In a similar fashion to equation (3.5) in \cite{KMP2} we have that 
\begin{equation*}
\dpp(T(\phi)c)=T\left(\frac{1}{2\pi i}\frac{d}{d\theta} \phi\right)c + T(\phi)\dpp(c).
\end{equation*}
Using this, it follows that
\begin{equation*}
\dpp^j(T(\phi)c)=\sum_k a_kb_k\quad(\textrm{finite sums})
\end{equation*} with $a_k$ bounded and $b_k\in I_s^\infty$.  So $\mnnorm{T(\phi)c}{N} < \infty$ and hence $T(\phi)c \in I_s^\infty$. A similar argument shows $cT(\phi)\in I_s^\infty$.

We further get estimates on $\mnnorm{T(\phi)c}{N}$ and $\mnnorm{cT(\phi)}{N}$, namely:
\begin{equation}\label{ctfestimate}
\mnnorm{T(\phi)c}{N} \leq \Cnorm{M}{\phi} \mnnorm{c}{N}\quad\textrm{and}\quad
\mnnorm{cT(\phi)}{N} \leq \Cnorm{M+N}{\phi} \mnnorm{c}{N}\,.
\end{equation}
These are simple for $M=0$.  Inductively, we have
\begin{align*}
\mnnorm[M+1]{T(\phi)c}{N} &= \mnnorm{T(\phi)c}{N} + \mnnorm{\dpp(T(\phi)c)}{N}
\\
&\leq \Cnorm{M}{\phi}\mnnorm{c}{N} + \mnnorm{\dpp(T(\phi))c}{N} + \mnnorm{T(\phi)\dpp(c)}{N} \\
&\leq \left(\Cnorm{M}{\phi}+\Cnorm{M}{\dfrac{1}{2\pi i}\frac{d}{d\theta} \phi}\right)(\mnnorm{c}{N}+\mnnorm{\dpp(c)}{N}) = \Cnorm{M+1}{\phi}\mnnorm[M+1]{c}{N}.
\end{align*}
For the second estimate, we use the derivation from equation \eqref{partialj}:
\begin{equation*}
\partial_N(a)=[(1+\PP)^N,a] = \sum_{j=1}^N\binom{N}{j}\dpp^j(a)(1+\PP)^{N-j}.
\end{equation*}
For the $M=0$ case of our estimate, we get
\begin{align*}
\mnnorm[0]{cT(\phi)}{N}&= \|cT(\phi)(1+\PP)^N\| = \|(1+\PP)^NT(\overline{\phi})c^*\| =\left\|\partial_N(T(\overline{\phi})c^*) +T(\overline{\phi})c^*(1+\PP)^N\right\| \\
&\leq \|\partial_N(T(\overline{\phi})c^*)\| + \|T(\overline{\phi})c^*(1+\PP)^N\| = \|\partial_N(T(\overline{\phi})c^*)\| + \|T(\overline{\phi})(1+\PP)^Nc^*\|  \\
\end{align*}
Using the formula for the commutator $\partial_N(a)$ we get:
\begin{align*}
\mnnorm[0]{cT(\phi)}{N}&\leq \left\|\sum_{j=1}^N \binom{N}{j} \dpp^j(T(\overline{\phi})c^*)(1+\PP)^{N-j}\right\| + \|T(\overline{\phi})(1+\PP)^Nc^*\| \\
&\leq \sum_{j=0}^N \binom{N}{j}\|\dpp^j(T(\overline{\phi}))(1+\PP)^{N-j}c^*\| = \sum_{j=0}^N\binom{N}{j} \left\|T\left(\left(\frac{1}{2\pi i}\frac{d}{d\theta}\right)^j \phi\right)\right\| \|c\|_{N-j} \\
&\leq \Cnorm{N}{\phi}\nnorm{c} = \Cnorm{N}{\phi}\mnnorm[0]{c}{N}
\end{align*}
Then for the inductive step, we see that
\begin{align*}
\mnnorm[M+1]{cT(\phi)}{N} &= \mnnorm{cT(\phi)}{N}+\mnnorm{\dpp(cT(\phi))}{N} \\
&\leq \mnnorm{c}{N}\Cnorm{M+N}{\phi}+\mnnorm{\dpp(c)}{N}\Cnorm{M+N}{\phi}+\mnnorm{c}{N}\Cnorm{M+N}{\left(\dfrac{1}{2\pi i}\frac{d}{d\theta}\phi\right)} \\
&\leq \Cnorm{M+1+N}{\phi}\mnnorm[M+1]{c}{N}
\end{align*}
Since all norms $\mnnorm{cT(\phi)}{N}$ and $\mnnorm{T(\phi)c}{N}$ are finite, the corresponding products are in $I_s^\infty$.

Let $\phi,\psi\in C^\infty(\R/\Z)$ with Fourier series
\begin{equation*}
\phi(\theta)=\sum_{n\in\Z} \phi_n e^{2\pi in\theta}\mbox{\ and\ } \psi(\theta)=\sum_{m\in\Z} \psi_m e^{2\pi im\theta}
\end{equation*} with $\{\phi_n\}, \{\psi_m\}$ RD.  Separate these into positive and negative frequency components:
\begin{equation}
\phi_+=\sum_{n\geq 0} \phi_n e^{2\pi in\theta} \mbox{\ and\ } \phi_-=\sum_{n<0} \phi_n e^{2\pi in\theta}
\end{equation}
and similarly for $\psi$. Clearly $\phi=\phi_++\phi_-$ and so $T(\phi)=T(\phi_+)+T(\phi_-)$; likewise for $\psi$.  A straightforward calculation shows that $T(\phi_+)T(\psi_+)=T(\phi_+\psi_+)$, $T(\phi_-)T(\psi_-)=T(\phi_-\psi_-)$, and $T(\psi_-)T(\phi_+)=T(\psi_-\phi_+)$.  It follows that
\begin{equation*}
T(\phi)T(\psi)-T(\phi\psi) = T(\phi_+)T(\psi_-)-T(\phi_+\psi_-)=T(\phi_+)T(\psi_-)-T(\psi_-)T(\phi_+)\,. 
\end{equation*}
For $n<0$, let $P_{<-n}$ be the orthogonal projection onto the subspace $\{E_{s^ml'}\}_{m<-n}$.   For $n<0$ and expanding $T(\psi_-)$ as 
\begin{equation*}
T(\psi_-) = \sum_{n<0}\psi_n(V^*)^{-n},
\end{equation*}
it follows that
\begin{equation*}
T(\phi)T(\psi)-T(\phi\psi) = -\sum_{n< 0} \psi_n (V^*)^{-n} T(\phi_+) P_{<-n}\,.
\end{equation*}

It is straightforward to see that $\dpp((V^*)^{-n})=n(V^*)^{-n}$ and that $\dpp(P_{<-n})=0$. Thus, applying $\dpp$ to $T(\phi)T(\psi)-T(\phi\psi)$ we obtain that:
\begin{equation*}
\dpp\left(T(\phi)T(\psi)-T(\phi\psi)\right)= \sum_{n<0} \psi_n (V^*)^{-n} \left(nT(\phi_+)+T\left(\left(\frac{1}{2\pi i}\frac{d}{d\theta}\right)\phi_+\right)\right) P_{<-n}\,.
\end{equation*}
Repeating this process $l$-times we get
\begin{equation}\label{dptftg-tfg}
\dpp^l(T(\phi)T(\psi)-T(\phi\psi))= -\sum_{n<0} \sum_{j=0}^l \binom{l}{j}\psi_n n^{l-j}(V^*)^{-n} T\left(\left(\frac{1}{2\pi i}\frac{d}{d\theta}\right)^j \phi_+\right) P_{<-n}\,.
\end{equation}

To estimate $M,N$-norm of the above, we first estimate the $N$-norm.   Using the triangle inequality, Proposition \ref{nprop}, and the fact that $\nnorm{P_{<-n}}=|n|^N$, we obtain:
\begin{equation*}
\begin{aligned}
\nnorm{\dpp^l(T(\phi)T(\psi)-T(\phi\psi))} &\le
\sum_{n<0}\sum_{j=0}^l \binom{l}{j}|\psi_n||n|^{l-j+N} \left\|T\left(\left(\frac{1}{2\pi i}\frac{d}{d\theta}\right)^j \phi_+\right)\right\|\\
&\le\sum_{n<0}|\psi_n|(1+|n|)^{l+N}\|\phi\|_{C^l}\,.
\end{aligned}
\end{equation*}
Using equation (3.7) in Section 3.3 from \cite{KMP3}, we get
\begin{equation*}
\nnorm{\dpp^l(T(\phi)T(\psi)-T(\phi\psi))}\le \left(\frac{\pi^2}{3}-1\right)\|\psi\|_{C^{l+N+2}}\|\phi\|_{C^l}\,.
\end{equation*}
Consequently, by the definition of the $M,N$-norm along with the following combinatorial estimate
\begin{align*}
&\sum_{l=0}^M \sum_{j=0}^l  \binom{M}{l} \binom{l}{j} |n|^{l-j} \left\|\left(\dfrac{1}{2\pi i} \dfrac{d}{d\theta}\right)^j \phi\right\|_\infty 
\leq \sum_{l=0}^M \sum_{j=0}^l \binom{M}{l-j}\binom{M}{j} |n|^{l-j} \left\| \left(\dfrac{1}{2 \pi i}\dfrac{d}{d\theta}\right)^j \phi\right\|_\infty\\
&\leq \sum_{j=0}^M \sum_{k=0}^M |n|^k \binom{M}{k}\binom{M}{j} \left\|\left(\dfrac{1}{2\pi i}\dfrac{d}{d\theta}\right)^j\phi\right\|_\infty = (1+|n|)^M \|\phi\|_{C^M}
\end{align*}
we get
\begin{equation*}
\|T(\phi)T(\psi)-T(\phi\psi))\|_{M,N}\le\left(\frac{\pi^2}{3}-1\right)\|\psi\|_{C^{M+N+2}}\|\phi\|_{C^M}<\infty
\end{equation*}
for all $M,N$.
\end{proof}

Following Proposition 3.2 from \cite{KMP3}, we can define a countable family of submultiplicative norms defining the topology on $HS^\infty(s)$ as follows:
\begin{equation*}
\|T(\phi)+c\|_{M,N}:=S\|\phi\|_{C^{M+N+2}}+\|c\|_{M,N}\,,
\end{equation*}
where $S$ is a sufficiently large constant.

\subsection{Generalized Toeplitz Operators}
Recall that for $\lambda\in C(\Zsstar)$, we defined a multiplication operator $M_\lambda:H\to H$ in equation \eqref{Mphi}. The key feature of those operators is that, while not in $HS(s)$, they commute with all elements of $HS(s)$ and if $c\in I_s$ then $M_\lambda c$ and $cM_\lambda$ are also in $I_s$.
This lets us introduce the concept of a ``generalized'' Toeplitz operator.  Consider the space
\begin{equation*}
C^\infty(\R/\Z\times\Zsstar) = \left\{\Lambda(\theta,x)=\sum_{n\in\Z}e^{2\pi in\theta}\lambda_n(x) : \lambda_n \in C(\Zsstar) \text{ and } \{\underset{x}{\textrm{sup }}|\lambda_n(x)|\}_n\textrm{ is RD}\right\}\,.
\end{equation*}
There are natural norms on $C^\infty(\R/\Z\times\Zsstar)$ given by:
\begin{equation*}
\Cnorm{k}{\Lambda}=\sum_{j=0}^k \binom{k}{j}\left\|\left(\dfrac{1}{2\pi i} \frac{d}{d\theta}\right)^j \Lambda\right\|_\infty,
\end{equation*}
where, for simplicity, we abuse the regular $C^k$ norm notation.  These norms are submultiplicative and satisfy recurrence:
\begin{equation*}
\Cnorm{k+1}{\Lambda}=\Cnorm{k}{\Lambda}+\Cnorm{k}{\dfrac{1}{2\pi i}\frac{d}{d\theta} \Lambda}
\end{equation*} 

Given $\Lambda\in C^\infty(\R/\Z\times\Zsstar)$ we can define a ``generalized'' Toeplitz operator by
\begin{equation*}
\mathcal{T}(\Lambda) = \sum_{n\ge0}V^nM_{\lambda_n} + \sum_{n<0}M_{\lambda_n}(V^*)^{-n}\,.
\end{equation*}
If $\Lambda$ does not depend on $x$ then $\mathcal{T}(\Lambda)$ is an actual Toeplitz operator, that is $\mathcal{T}(\Lambda)=T(\Lambda)$.  Conversely, if $\mathcal{T}(\Lambda)$ is an actual Toeplitz operator then $\Lambda$ does not depend on $x$. Because of the properties of $M_\lambda$ for $\lambda\in C(\Zsstar)$ mentioned above, Proposition \ref{IVIdeal} immediately generalizes to the following result, with the same proof.

\begin{prop}\label{fancytoeplitz} If $\Lambda\in C^\infty(\R/\Z\times\Zsstar)$ and $c\in I_s^\infty$, then 
\begin{equation*}
\mnnorm{\mathcal{T}(\Lambda)c}{N} \leq \Cnorm{M}{\Lambda} \mnnorm{c}{N}\quad\textrm{and}\quad
\mnnorm{c\mathcal{T}(\Lambda)}{N} \leq \Cnorm{M+N}{\Lambda} \mnnorm{c}{N}\,.
\end{equation*}
If $\Lambda$ and $\Omega$ are in $C^\infty(\R/\Z\times\Zsstar)$, then
\begin{equation*}
\|\mathcal{T}(\Lambda)\mathcal{T}(\Omega)-\mathcal{T}(\Lambda\Omega))\|_{M,N}\le\left(\frac{\pi^2}{3}-1\right)\|\Omega\|_{C^{M+N+2}}\|\Lambda\|_{C^M}.
\end{equation*}
\end{prop}

\subsection{Holomorphic Stability of $HS^\infty(s)$} With extra manipulations of Toeplitz operators, the stability of $HS^\infty(s)$ under holomorphic functional calculus follows from the corresponding stability of $I_s^\infty$.
\begin{prop}
If $a\in HS^\infty(s)$ and $a$ is invertible in $HS(s)$, then $a^{-1}\in HS^\infty(s)$.
\end{prop}

\begin{proof} 
Since $a\in HS^\infty(s)$, $a=T(\phi)+b$ for some $\phi\in C^\infty(\R/\Z)$ and $b\in I_s^\infty$.  Since the quotient map $q(a)=\phi$ is a homomorphism, $q(a^{-1})=1/\phi$.  Since $\phi\neq 0$, $1/\phi$ is in $C^\infty(\R/\Z)$.  Write $a^{-1}=T(1/\phi)+\tilde{b}$ for some $\tilde{b}$.  We will show that $\tilde{b}\in I_s^\infty$.

Rearranging terms we get 
\begin{align*}
\tilde{b}&=a^{-1}-T(1/\phi) = a^{-1}(I-aT(1/\phi)) \\
&= a^{-1}(I-(T(\phi)+b)T(1/\phi)) \\
&= a^{-1}(I-T(\phi)T(1/\phi)+bT(1/\phi))
\end{align*}
By Proposition \ref{IVIdeal}, $I-T(\phi)T(1/\phi)\in I_s^{\infty}$ and $bT(1/\phi)\in I_s^\infty$, so consequently, there is a $\tilde{c}\in I_s^\infty$ such that $\tilde{b}=a^{-1}\tilde{c}$.  It follows from Proposition \ref{nprop} that
\begin{equation}\label{inverse_est}
\|\tilde{b}\|_{0,N}\le\|a^{-1}\|\|\tilde{c}\|_{0,N}<\infty\,.
\end{equation}
Computing $\dpp$ on $\tilde{b}$ we have
\begin{equation*}
\dpp(\tilde{b}) = \dpp(a^{-1})\tilde{c} + a^{-1}\dpp(\tilde{c}) = -a^{-1}\dpp(a)a^{-1}\tilde{c} + a^{-1}\dpp(\tilde{c})\,.
\end{equation*}
So, as in the proofs of Proposition  \ref{IpluscInv}  or  Proposition \ref{IVIdeal}, we have, inductively, for any $j$ that
\begin{equation*}
\dpp^j(\tilde{b}) = \sum_ia_ib_i \quad\textrm{finite sum,}
\end{equation*}
with $a_i$ bounded and $b_i$ are in $I_s^{\infty}$.  Using this and the estimate in equation \eqref{inverse_est}, we see that $\|\tilde{b}\|_{M,N}$ is finite for all $M$ and $N$.  Thus $\tilde{b}\in I_s^\infty$, completing the proof.

\end{proof}

\begin{theo}\label{HSholFC}
The smooth Hensel-Steinitz algebra $HS^\infty(s)$ is closed under the holomorphic functional calculus. In other words, for any $a\in HS^\infty(s)$ and a function $g$ that is holomorphic on an open neighborhood containing the spectrum of $a$ we have $g(a)\in HS^\infty(s)$.
\end{theo}

\begin{proof}
If $C$ is a contour around an open set containing the spectrum of $a$ then the Cauchy integral formula gives us
\[g(z)=\dfrac{1}{2\pi i}\int_C \dfrac{g(\zeta)}{\zeta-z}\,d\zeta\]
Then we define $g(a)$ by
\[g(a)=\dfrac{1}{2\pi i}\int_C g(\zeta)(\zeta-a)^{-1}\,d\zeta \]

It is thus enough to show that if $a$ is invertible in $HS(s)$ and $a\in HS^{\infty}(s)$ then $a^{-1}\in HS^\infty(s)$. However, this was established in the previous proposition.
\end{proof}

\subsection{Smooth Stability of $HS^\infty(s)$} To establish stability of $HS^{\infty}(s)$ under smooth calculus of self-adjoint elements we repeat similar ideas from the study of stability of $I_s^\infty$ and study exponentials $e^{ia}$ of self-adjoint elements $a\in HS^\infty(s)$. We first identify the Toeplitz operator part of the exponentials and then estimate the difference using essentially Duhamel's principle. 

The key inequality to estimate the $M,N$-norms of products, Proposition \ref{mnprop} part $(4)$, requires a better understanding of the $M,0$-norms of exponentials, which is done in the next statement. We follow reference \cite{KMP3}, Proposition 3.5.

%

\begin{prop}\label{eicestimate} 
Suppose $c\in I_s^\infty$ is a self-adjoint element.  Then we have the following inequality:
\begin{equation*}
\mnnorm{e^{ic}}{0} \leq \prod_{j=1}^M(1+\mnnorm[j]{c}{0})^{2^{M-j}}\,.
\end{equation*}
\end{prop}

\begin{proof}
We prove this via induction. For $M=0$ it is clear that $\mnnorm[0]{e^{ic}}{0}=1$.  From Proposition \ref{mnprop} part (2) we have that
\begin{equation*}
\mnnorm[M+1]{e^{ic}}{0} = \mnnorm{e^{ic}}{0}+\mnnorm{\dpp(e^{ic})}{0}
\end{equation*}
Notice that we have
\begin{equation*}
\dpp(e^{ic})=\int_0^1 \frac{d}{dt}\left(e^{i(1-t)c}(1+\PP)e^{itc}\right)\,dt 
=i\int_0^1 e^{i(1-t)c}\dpp(c)e^{itc}\,dt\,.
\end{equation*}
Assume the inequality is true for $M$ and consider the following: 
\begin{align*}
\mnnorm[M+1]{e^{ic}}{0} &= \mnnorm{e^{ic}}{0}+\mnnorm{\dpp(e^{ic})}{0} \\
&\leq \prod_{j=1}^M(1+\mnnorm[j]{c}{0})^{2^{M-j}} + \mnnorm{\int_0^1 e^{i(1-t)c}\dpp(c)e^{itc}\,dt}{0}\\
&\leq \prod_{j=1}^M(1+\mnnorm[j]{c}{0})^{2^{M-j}} + \int_0^1 \|e^{i(1-t)c}\|_{M,0}\|\delta_{\mathbb P}(c)\|_{M,0}\|e^{itc}\|_{M,0}\,dt.
\end{align*}
Using the inductive assumption again we get
\begin{align*}
\mnnorm[M+1]{e^{ic}}{0} &
\leq\prod_{j=1}^M(1+\mnnorm[j]{c}{0})^{2^{M-j}}\\
&\qquad+\mnnorm{\dpp(c)}{0}\int_0^1\prod_{j=1}^M(1+(1-t)\mnnorm[j]{c}{0})^{2^{M-j}}\prod_{j=1}^M(1+t\mnnorm[j]{c}{0})^{2^{M-j}}\,dt \\
&\leq \prod_{j=1}^M(1+\mnnorm[j]{c}{0})^{2^{M-j}} + \left[\prod_{j=1}^M(1+\mnnorm[j]{c}{0})^{2^{M-j}}\right]^2 \mnnorm{\dpp(c)}{0}\,. 
\end{align*}
By Proposition \ref{mnprop}, we have the inequality 
\begin{equation*}
\mnnorm{\dpp(c)}{0}\leq \mnnorm[M+1]{c}{0},
\end{equation*} 
and thus we obtain the desired result:
\begin{align*}
\mnnorm[M+1]{e^{ic}}{0} &\leq \prod_{j=1}^M(1+\mnnorm[j]{c}{0})^{2^{M-j}} \left(1+\prod_{j=1}^M(1+\mnnorm[j]{c}{0})^{2^{M-j}}\mnnorm[M+1]{c}{0}\right) \\
&\leq \prod_{j=1}^M(1+\mnnorm[j]{c}{0})^{2^{M-j}} \prod_{j=1}^M(1+\mnnorm[j]{c}{0})^{2^{M-j}} \left(1+\mnnorm[M+1]{c}{0}\right) \\
&=\prod_{j=1}^{M+1}(1+\mnnorm[j]{c}{0})^{2^{M+1-j}}\,.
\end{align*}
\end{proof}

In order to prove the next theorem we need the following inequality.  
If $\phi\in C^\infty(\R/\Z)$ is real-valued, then we have the following estimate
\begin{equation}\label{eifestimate}
\Cnorm{M}{e^{i\phi}}\leq \prod_{j=1}^M(1+\|\phi\|_{C^j}).
\end{equation}
This was proved in Proposition 3.4 in \cite{KMP3}.

\begin{theo}\label{BvinfsmoothFC}
The algebra $HS^\infty(s)$ is closed under the smooth functional calculus for self-adjoint elements.  That is, if $a\in HS^\infty(s)$ is self-adjoint and $g$ is smooth on an open neighborhood of the spectrum of $a$ then $g(a)\in HS^\infty(s)$.
\end{theo}

\begin{proof}

Extend $g$ to be smooth, $L$-periodic in a neighborhood of the spectrum of $a$; then $g(\theta)=g(\theta+L)$.  

Then for $a=T(\phi)+c\in HS^\infty(s)$, we have $g(a)=\sum_{n\in\Z}g_n e^{2\pi ina/L}$.  Since $\mnnorm{\cdot}{N}$ is submultiplicative, we get $e^{2\pi i na/L} \in HS^\infty(s)$ for each $n$.  The sum will be in $HS^\infty(s)$ if there is a polynomial bound in $n$ for $\mnnorm{e^{2\pi ina/L}}{N}$.

Since for the quotient map $q:HS(s)\to C(\R/\Z)$ we have $q(a)=\phi$ and since $q$ is a homomorphism it follows that $q(e^{2\pi ina/L})= e^{2\pi in\phi/L}$.  Note that norms of $e^{2\pi i n \phi}$  grow at most polynomially in $n$ from equation \eqref{eifestimate}.  Thus we need to study the $M,N$-norms of the difference $e^{2\pi ina/L}-T(e^{2\pi in\phi/L}) \in I_s^\infty$ and show that they grow at most polynomially in $n$.  To do this, we again use essentially Duhamel's principle.

\begin{align*}
e^{i(T(\phi)+c)}-T(e^{i\phi}) &= \int_0^1 \frac{d}{dt}\left(e^{it(T(\phi)+c)}T(e^{i(1-t)\phi})\right)\,dt \\
&=i\int_0^1 e^{it(T(\phi)+c)}cT(e^{i(1-t)\phi})\,dt \\ 
&\qquad + \int_0^1 e^{it(T(\phi)+c)}\left(T(\phi)T(e^{i(1-t)\phi}) - T(fe^{i(1-t)\phi})\right)\,dt\,.
\end{align*}
Applying the $M,N$-norm to this and using Proposition \ref{mnprop} yields: 
\begin{align*}
\mnnorm{e^{i(T(\phi)+c)}-T(e^{i\phi})}{N} 
&\leq \int_0^1 \mnnorm{e^{it(T(\phi)+c)}}{0}\mnnorm{cT(e^{i(1-t)\phi})}{N} \,dt \\
&\qquad +\int_0^1 \mnnorm{e^{it(T(\phi)+c)}}{0}\mnnorm{T(\phi)T(e^{i(1-t)\phi}) - T(fe^{i(1-t)\phi})}{N}\,dt\,.
\end{align*}
Using equation \eqref{eifestimate} and Proposition \ref{eicestimate}  we have
\begin{align*}
&\mnnorm{e^{i(T(\phi)+c)}-T(e^{i\phi})}{N} \leq \left(\prod_{j=1}^M(1+\Cnorm{j}{\phi}+\mnnorm[j]{c}{0})^{2^{M-j}}\right) \mnnorm{c}{N}\left( \prod_{j=1}^M (1+\Cnorm{j}{\phi})\right) \\
&\quad +\left(\prod_{j=1}^M(1+\Cnorm{j}{\phi}+\mnnorm[j]{c}{0})^{2^{M-j}}\right)\left(\frac{\pi^2}{3}-1\right) \Cnorm{M}{\phi} \left(\prod_{j=1}^{M+N+2}(1+\Cnorm{j}{\phi})\right)\,.
\end{align*}
This yields the desired growth estimate in $n$.
\end{proof}

\section{Derivations}

The purpose of this section is to describe continuous derivations on $HS^\infty(s)$. The main result, Theorem \ref{der_theo},  exhibits a geometric structure of the space of derivations modulo inner derivations. This is achieved through a series of steps. We first observe that continuous derivations preserve the ideal  $I_s^\infty$ and thus define derivations on the factor algebra $HS^\infty(s)/I_s^\infty\cong C^\infty(\R/\Z)$. Conversely, we show that any derivation on $C^\infty(\R/\Z)$ can be lifted to a derivation on $HS^\infty(s)$. This observation reduces the study of derivations to those with range in $I_s^\infty$. We analyze continuous derivations $HS^\infty(s)\to I_s^\infty$, the key part of classification of derivations, by using Fourier components with respect to the natural circle action given by $\rho_\theta$ defined in equation \eqref{rhotheta}.
\subsection{Invariance of $I_s^\infty$ under Derivations}

\begin{prop}\label{DerivationsIdeals}
Let $\delta:HS^\infty(s)\to HS^\infty(s)$ be a continuous derivation.  Then we have the following inclusion: $\delta(I_s^\infty)\subseteq I_s^\infty$.
\end{prop}

\begin{proof}
We first need to show that $\delta(P_{ij}M_\lambda)\in I_s^{\infty}$ for $\delta$ a continuous derivation, where $P_{ij}$ are given by equation \eqref{Pij}, and $M_\lambda$ is given by equation \eqref{Mphi}.  

Using Leibniz rule we have that
\begin{equation*}
\delta(P_{ik})=\delta(P_{ij}P_{rk}) = P_{ij}\delta(P_{rk})+\delta(P_{ij})P_{rk}\in I_s^\infty
\end{equation*}
since $I_s^\infty$ is an ideal. 
Next notice that 
\begin{equation*}
\delta(P_{ij}M_\lambda)=\delta(P_{ij})M_\lambda + P_{ij}\delta(M_\lambda) \in I_s^{\infty}
\end{equation*}
because of the above calculation and the fact that $I_s^{\infty}$ is an ideal.
Moreover, finite linear combinations of $P_{ij}M_\lambda$ are elements of $I_s^\infty$ and since such combinations are dense in $\mathcal{K}\otimes C(\Zsstar)$ and $\delta$ is continuous, it follows that $\delta(I_s^\infty)\subseteq I_s^\infty$.
\end{proof}

Any continuous derivation $\delta: HS^\infty(s)\to HS^\infty(s)$ defines a derivation $[\delta]$ in the factor algebra $[\delta]:C^\infty(\R/\Z)\to C^\infty(\R/\Z)$ given by 
\begin{equation*}
[\delta](a+I_s^\infty) = \delta(a)+I_s^\infty\,.
\end{equation*}
This is well defined, as if $a$ and $a'$ are such that $[a]=[a']$ then $a-a'\in I_s^\infty$ and Proposition \ref{DerivationsIdeals} implies that $\delta(a-a')\in I_s^\infty$, so $[\delta(a)]=[\delta(a')]$.  Since $\delta$ is a derivation on $HS^\infty(s)$, it follows that $[\delta]$ is a derivation on $C^\infty(\R/\Z)$.  General theory, see \cite{newns}, implies that derivations on smooth functions on manifolds are smooth vector fields, so there is an $\varphi\in C^\infty(\R/\Z)$ such that
\begin{equation}\label{classofdelta}
[\delta]=\varphi(\theta)\frac{1}{2\pi i}\frac{d}{d\theta}\,.
\end{equation}
Below we show that any derivation on $C^\infty(\R/\Z)$ can be lifted to a derivation on $HS^\infty(s)$.

For an $\varphi\in C^\infty(\R/\Z)$, decompose it as $\varphi=\varphi_+ +\varphi_-$, where
\begin{equation*}
\varphi_+(\theta) = \sum_{n\ge0}\varphi_ne^{2\pi in\theta}\quad\textrm{and}\quad \varphi_-(\theta) = \sum_{n<0}\varphi_ne^{2\pi in\theta}\,.
\end{equation*}

\begin{prop}\label{special_der_delta_F}
The formula 
\begin{equation*}
\delta_\varphi(a) = [T(\varphi_+)\mathbb{P} +\mathbb{P} T(\varphi_-),a]
\end{equation*}
defines a continuous derivation on $HS^\infty(s)$ so that the following hold for the quotient map:
\begin{equation*}
{q}(\delta_\varphi(V)) = \varphi(\theta)e^{2\pi i\theta}\quad\textrm{and}\quad {q}(\delta_\varphi(V^*))=-\varphi(\theta)e^{-2\pi i\theta}\,.
\end{equation*}
\end{prop}

\begin{proof}
That $\delta_\varphi$ is continuous is straightforward: multiplication by a Toeplitz operator $T(\varphi)$ or $\PP$ is continuous, as are sums of such.

For a given $\varphi \in C^\infty(\R/\Z)$ and writing $\varphi = \varphi_+ + \varphi_-$, we have the following:
\begin{equation*}
\delta_\varphi(V)= T(\varphi_+) \PP V -VT(\varphi_+)\PP + \PP T(\varphi_-) V - V\PP T(\varphi_-)\,.
\end{equation*} 
Using the identities $\PP VV^*=\PP$ and $[\PP,V]=V$ and the fact that $V$ commutes with $T(\varphi_+)$, we have
\begin{equation*}
\delta_\varphi(V)=VT(\varphi_+) + \PP T(\varphi_-) V - V\PP T(\varphi_-)\,.
\end{equation*}
Using the Fourier series expansion of $\varphi$, we have the following calculation:
\begin{equation*}
\begin{aligned}
\PP T(\varphi_-) V - V\PP T(\varphi_-) &= \sum_{n<0} \varphi_n \PP (V^*)^{-n} V - V\PP \sum_{n<0} \phi_n(V^*)^{-n} \\
&= \sum_{n<0} \varphi_n \PP (V^*)^{-n-1}  - V\PP \sum_{n<0} \phi_n(V^*)^{-n} \\
&=  \sum_{n<0} \varphi_n \PP V (V^*)^{-n}  - V\PP \sum_{n<0} \phi_n(V^*)^{-n}
\end{aligned}
\end{equation*}
where we used the fact that $\PP VV^*=\PP$ again.  It now follows that
\begin{equation*}
\PP T(\varphi_-) V - V\PP T(\varphi_-) = V T(\varphi_-)\,.
\end{equation*}
Combining this with the above calculation we get that
\begin{equation*}\delta_\varphi(V)=V T(\varphi_+)+VT(\varphi_-) = VT(\varphi)\,.
\end{equation*}

Under the mapping $q$ we have 

\begin{equation*}
q(\delta_\varphi(V))= q(VT(\varphi)) = e^{2\pi i \theta}\varphi(\theta) = \varphi(\theta)\dfrac{1}{2\pi i}\dfrac{d}{d\theta} e^{2\pi i\theta} = [\delta_\varphi](q(V)).
\end{equation*}
\end{proof}

It follows by continuity that for the derivation $\delta_\varphi$ defined in the above proposition, given an $a\in HS^\infty(s)$, we have for the following formula for the quotient
\begin{equation*}
{q}(\delta_\varphi(a)) = \varphi(\theta)\frac{1}{2\pi i}\frac{d}{d\theta}{q}(a)(\theta)\,.
\end{equation*}
Hence, $\delta_\varphi$ is a lift of the derivation
\begin{equation*}
\varphi(\theta)\frac{1}{2\pi i}\frac{d}{d\theta}
\end{equation*}
on $C^\infty(\R/\Z)$.

\subsection{Fourier Components of Derivations}
The main tool to study derivations in $HS^\infty(s)$ is a version of Fourier transform. This is because the corresponding Fourier coefficients are derivations with additional covariance properties that makes it possible to fully classify them.

Given $n\in\Z$, a derivation $\delta:HS^\infty(s)\to HS^\infty(s)$ is said to be a {\it $n$-covariant derivation} if the relation 
\begin{equation*}
\rho_\theta^{-1}\delta\rho_\theta(a)= e^{-2\pi in\theta} \delta(a)
\end{equation*}
holds for all $\theta\in\R/\Z$, where $\rho_\theta$ was defined in equation \eqref{rhotheta}.  When $n=0$ we say that the derivation is invariant.  With this definition, we point out that $\delta_\mathbb{P}:HS^\infty(s)\to HS^\infty(s)$ is an invariant continuous derivation.

\begin{defin}
If $\delta$ is a continuous derivation in $HS^\infty(s)$, the {\it $n$-th Fourier component} of $\delta$ is defined as: 
\begin{equation*}
\delta_n(a)= \int_0^1 e^{2\pi in\theta} \rho_\theta^{-1}\delta\rho_\theta(a)\, d\theta\,.
\end{equation*}
\end{defin}
Below we describe properties of Fourier components of derivations and of $n$-covariant derivations in general.

\begin{prop}\label{ncovariantdeltan}
Let $\delta:HS^\infty(s)\to HS^\infty(s)$ be a continuous derivation.  Then $\delta_n:HS^\infty(s)\to HS^\infty(s)$ is a continuous $n$-covariant derivation, where $\delta_n$ are the $n$-th Fourier components of $\delta$. Moreover,  $\delta=0$ if and only if $\delta_n=0$ for all $n\in\Z$.
\end{prop}

\begin{proof}
Since $\delta:HS^\infty(s)\to HS^\infty(s)$, $\rho_\theta:HS^\infty(s)\to HS^\infty(s)$ and $HS^\infty(s)$ is complete, it follows that $\delta_n(a)\in HS^\infty(s)$ for all $a\in HS^\infty(s)$. Since $\delta$ is a continuous derivation and the automorphism $\rho_\theta$ is continuous, it follows that $\delta_n$ is also continuous, see also the proof of the lemma below.  It is straightforward to see that $\delta_n$ is a derivation.

The following computation verifies that $\delta_n$ is $n$-covariant:
\begin{equation*}
\rho_\theta^{-1}\delta_n\rho_\theta(a) = \int_0^1 e^{2\pi in\varphi} \rho_\theta^{-1}\rho_\varphi^{-1}\delta\rho_\varphi\rho_\theta(a)\, d\varphi = \int_0^1 e^{2\pi in\varphi}\rho_{\theta + \varphi}^{-1}\delta\rho_{\theta + \varphi}(a)\, d\varphi\,.
\end{equation*}
Changing to new variable $\theta + \varphi$, and using the translation invariance of the measure, it now follows that $\rho_\theta^{-1}\delta_n\rho_\theta(a)= e^{-2\pi in\theta} \delta_n(a)$.

The usual Ces\`aro mean convergence result for Fourier components in harmonic analysis \cite{K} implies that if $\delta$ is a continuous derivation on $HS^\infty(s)$ then
\begin{equation}\label{deltacesaro}
\delta(a)=\lim_{L\rightarrow\infty} \frac{1}{L+1} \sum_{j=0}^L \left(\sum_{n=-j}^j \delta_n(a)\right)\,,
\end{equation}
for every $a\in HS^\infty(s)$.
In particular, $\delta$ is completely determined by its Fourier components $\delta_n$.
This shows that $\delta =0$ if and only if $\delta_n=0$ for all $n\in \Z$.

\end{proof}

We will also need the following growth estimate on Fourier components of derivations.
\begin{lem}\label{est_delta_n}
Let $\delta:HS^\infty(s)\to HS^\infty(s)$ be a continuous derivation.  Then, for every $k\ge0$, $M\ge0$ and $N\ge0$, there exist $M'\ge0$ and $N'\ge0$ and a constant $C_k=C_k(M,N)$ such that for every $a\in I_s^\infty$ we have:
\begin{equation*}
n^k\|\delta_n(a)\|_{M,N} \le C_k\|a\|_{M',N'}\,,
\end{equation*}
where $\delta_n$ are the $n$-th Fourier components of $\delta$.
\end{lem}
\begin{proof}
Since $\delta$ is a continuous linear map on the Frech\'et space $HS^\infty(s)$ for every $M$ and $N$ there exist constants $C=C(M,N)$, $M',N'\ge 0$ such that for every $a\in I_s^\infty$, 
\begin{equation*}
\|\delta(a)\|_{M,N} \le C\|a\|_{M',N'}.
\end{equation*}
Using this and the fact that $\rho_\theta$ is continuous, we have 
\begin{equation}\label{normest}
\|\delta_n(a)\|_{M,N} = \left\| \int_0^1 e^{2\pi in\theta} \rho_\theta^{-1}\delta\rho_\theta(a)\,d\theta\right\|_{M,N} \le \int_0^1 \|\rho_\theta^{-1}\delta\rho_\theta(a)\|_{M,N}\,d\theta \leq C\|a\|_{M',N'}.
\end{equation}
Consider the following calculation, using integration by parts, the continuity of $[\dpp,\delta]$, and the continuous differentiability of $\theta \to \rho_\theta^{-1}\delta\rho_\theta$:
\begin{equation*}
\begin{aligned}
2\pi i n \delta_n(a) &= \int_0^1 2\pi ine^{2\pi in\theta}\rho_\theta^{-1}\delta\rho_\theta(a)\,d\theta =\int_0^1 \dfrac{d}{d\theta}(e^{2\pi in\theta}) \rho_\theta^{-1}\delta\rho_\theta(a)\,d\theta \\
&=-\int_0^1 e^{2\pi i n \theta}\dfrac{d}{d\theta} (\rho_\theta^{-1}\delta\rho_\theta(a))\,d\theta = -2\pi i \int_0^1\rho_\theta^{-1}[\dpp,\delta]\rho_\theta(a)\,d\theta \\
&=-2\pi i([\dpp,\delta])_n(a),
\end{aligned} 
\end{equation*}
where $([\dpp,\delta])_n$ is the $n$-th Fourier component of the derivation $[\dpp,\delta]$.  Consequently, using estimate \eqref{normest} for the commutator $[\dpp,\delta]$ we obtain
\begin{equation*}
n\|\delta_n(a)\|_{M,N} \le C\|a\|_{M',N'},
\end{equation*}
for some $M',N'$ and constant $C$.  The result now follows easily by induction on $k$.
\end{proof}

\begin{prop}\label{norm_conv}
Let $\delta: HS^\infty(s)\to HS^\infty(s)$ be a continuous derivation, then
\begin{equation}\label{der_series}
\delta(a) = \sum_{n\in\Z} \delta_n(a)
\end{equation}
for every $a\in I_s^\infty$ where $\delta_n$ are the $n$-th Fourier components of $\delta$.  Here the convergence is with respect to the $M,N$-norms.  In particular, this implies norm convergence in $HS(s)$.
\end{prop}

\begin{proof}
Lemma \ref{est_delta_n} implies that $\{\|\delta_n(a)\|_{M,N}\}$ is a RD sequence for every $a\in I_s^\infty$ and thus the series given in equation  \eqref{der_series} is norm convergent. Also, the right-hand side of equation \eqref{der_series} is a continuous derivation on $HS^\infty(s)$ with the same Fourier components $\delta_n$ as $\delta$ so they must be equal.
\end{proof}  

\subsection{Derivations with Range in $I_s^\infty$}
We now turn to study of derivations with range in $I_s^\infty$.
Recall that the Fourier coefficients of $a\in I_s^\infty$
are in the space
\begin{equation*}
C_0^\infty(\Z_{\ge0}\times\Zsstar)=\left\{F\in C_0(\Z_{\ge0}\times\Zsstar):  \{\underset{x}{\textrm{sup }}|F(m,x)|\}_m\textrm{ is RD}\right\}\,.
\end{equation*}
Using the following identification
\begin{equation*}
\{f\in C(\Z_s):f(0)=0\} \cong  C_0(\Z_{\ge0}\times\Zsstar)
\end{equation*}
we can extend the endomorphisms $\alpha$ and $\beta$ to  $C_0^\infty(\Z_{\ge0}\times\Zsstar)$ by
\begin{equation}\label{alphabetatilde}
(\beta F)(m,x)=F(m+1,x)\quad\textrm{and}\quad
 (\alpha F)(m,x)= \left\{
\begin{aligned}
&F(m-1,x) &&\textrm{if }m> 0 \\
&0 &&\textrm{else,}
\end{aligned}\right.
\end{equation}
where we abused notation for readability purposes.

\begin{prop}\label{covariant_inner}
Let $\delta:HS^\infty(s)\to I_s^\infty$ be a continuous $n$-covariant derivation.  Then there exists $R \in C_0^\infty(\Z_{\ge0}\times\Zsstar)$ and $r\in C(\Zsstar)$ such that
\begin{equation*}
\delta(a)=\left\{
\begin{aligned}&[V^n(M_{R}+M_r),a] &\text{ if } n\ge 0\\
&[(M_{R}+M_{r})(V^*)^{-n},a] &\text{ if } n<0&.
\end{aligned}
\right.
\end{equation*}

\end{prop}

\begin{proof}
We first consider the case when $n=0$ or $\delta$ is invariant. 
 Since $\delta$ is invariant, we must have
\begin{equation*}
\rho_\theta(\delta(M_f))=\delta(\rho_\theta(M_f))=\delta(M_f)
\end{equation*}  
for every $f\in C(\Z_s)$.  Thus $\delta(M_f)$ is $\rho_\theta$-invariant for every $\theta$.
Note the following observation:  
\begin{equation*}
\{a\in HS(s): \forall\theta, \rho_\theta(a)=a \} = C^*(M_g: g\in C(\Z_s)).
\end{equation*}
Consequently, we have
\begin{equation*}
\delta(M_f)\in C^*(M_g:g\in C(\Z_s)).
\end{equation*}
Therefore $\delta$ defines a derivation on $C^*(M_g: g\in C(\Z_s) )$, however $C^*(M_g:g\in C(\Z_s))$ is abelian and there are no non-trivial derivations on abelian C$^*$-algebras, see \cite{S}.  Therefore $\delta(M_f)=0$.

Consider the following calculation:
\begin{equation*}
\begin{aligned}
\rho_\theta(\delta(V)V^*)&= \rho_\theta(\delta(V))\rho_\theta(V^*) = \delta(e^{2\pi i\theta}V)e^{-2\pi i\theta}V^* = \delta(V)V^*. 
\end{aligned}
\end{equation*}
From the above observation, $\delta(V)V^* = M_f$ for some $f\in C(\Z_s)$.  It follows from the commutation relation that 
\begin{equation*}
\delta(V) = M_fV = VM_{\beta f} :=V M_F
\end{equation*} 
with $F\in C(\Z_s)$.  Since $V\in HS^\infty(s)$ we have that $\delta(V)\in I_s^\infty$ and so $F\in C_0^\infty(\Z_{\ge0}\times \Z_s^\times)$. 

By a similar argument, $\delta(V^*)=M_G V^*$ with $G\in C_0^\infty(\Z_{\ge0}\times \Z_s^\times)$. Consider the following calculation:
\begin{equation*}
0= \delta(I) = \delta(V^*V)= \delta(V^*)V + V^*\delta(V) = M_GV^*V + V^*VM_F.
\end{equation*}
It follows that $G=-F$.  By continuity, $\delta$ is determined by $F$.  

We are looking for an $R\in C_0^\infty(\Z_{\ge0}\times \Z_s^\times)$ so that 
\begin{equation*}
\delta(a)=[M_R,a]\,.
\end{equation*}
Computing directly we have
\begin{equation*}
VM_F=\delta(V)=[M_R,V]=M_RV-VM_R=V(M_{\beta R}-M_R)
\end{equation*}
where $\beta$ is defined in equation \eqref{alphabetatilde}.  This yields the equation:
\begin{equation*}
F(m,x)=R(m+1,x)-R(m,x)\,.
\end{equation*}
Solving recursively gets
\begin{equation*}
R(m,x)=-\sum_{j=m+1}^\infty F(j,x)\,.
\end{equation*}
This sum converges since the sequence $\left\{\sup_x |F(j,x)|\right\}_j$ is rapid decay.   For any $N$, we have the following estimate
\begin{equation*}
m^N|R(m,x)| \leq \sum_{j=m}^\infty m^N|F(j,x)| \leq \sum_{j=m}^\infty j^N|F(j,x)|\,.
\end{equation*}
The right hand side of the above inequality goes to zero as $m\to\infty$ since $\{j^N|F(j,x)|\}$ is still a RD sequence.  It follows that $R(m,x)$ is RD in $m$ and hence $\delta$ has the desired form.
%

Next suppose that $n>0$, and since $\delta$ is assumed to be $n$-covariant we have that
\begin{equation*}
\rho_\theta(\delta(a))=e^{2\pi in\theta}\delta(\rho_\theta(a))
\end{equation*}
for all $a\in HS^\infty(s)$.  Argued similarly to the invariant case above we have that
\begin{equation*}
\delta(V)=V^{n+1}M_F\quad\textrm{and}\quad \delta(V^*)=V^{n-1}M_G
\end{equation*}
for some $F,G\in C_0^\infty(\Z_{\ge0}\times \Z_s^*)$. Moreover, we have
\begin{equation*}
0=\delta(V^*V)=V^{n-1}M_GV+V^*V^{n+1}M_F=V^n(M_{\beta G+F})\,.
\end{equation*}
So we have that $F=-\beta G$.  

Given $f\in C^\infty(\R/\Z)$, since $M_f$ is invariant under $\rho_\theta$ it follows that $\delta(M_f)=V^nM_{\tilde{\delta}(f)}$ where $\tilde{\delta}:C^\infty(\R/\Z)\to C^\infty(\R/\Z)$ is a linear map.  Since $a\in HS^\infty(s)$ we can write $a=T(\varphi)+c$ with $\varphi\in C^\infty(\R/\Z)$ and $c\in I_s^\infty$, and so if $M_f\in HS^\infty(s)$, then $M_f=f_0+M_F$ with $f_0\in \C$ and $F\in C_0^\infty(\Z_{\ge0}\times \Zsstar)$.  Since $\delta(f_0)=0$ we only need to study
\begin{equation*}
\delta(M_F)=V^nM_{\tilde{\delta}(F)}
\end{equation*}
with $\tilde{\delta}:C_0^\infty(\Z_{\ge0}\times\Zsstar)\to C_0^\infty(\Z_{\ge0}\times \Zsstar)$.  

From the definition of $\tilde{\delta}$ and the Leibniz rule we have
\begin{equation*}
V^nM_{\tilde{\delta}(F)}M_G + M_FV^nM_{\tilde{\delta}(G)} = \delta(M_FM_G)=\delta(M_{FG})
\end{equation*}
which yields
\begin{equation*}
\delta(M_{FG})=V^n(M_{\tilde{\delta}(F)G +\beta^n(F)\tilde{\delta}(G)})\,.
\end{equation*}
This gives a ``twisted'' derivation $\tilde{\delta}$ with the following product rule:
\begin{equation*}
\tilde{\delta}(FG)= \tilde{\delta}(F)G + \beta^n(F)\tilde{\delta}(G)\,.
\end{equation*}
Since $C_0^\infty(\Z_{\ge0}\times \Zsstar)$ is a commutative algebra, we get
\begin{equation*}
\tilde{\delta}(F)G + \beta^n(F)\tilde{\delta}(G) = \tilde{\delta}(G)F+\beta^n(G)\tilde{\delta}(F)
\end{equation*}
which gives
\begin{equation*}
\tilde{\delta}(F)(G-\beta^n(G))=\tilde{\delta}(G)(F-\beta^n(F))
\end{equation*}
for all $F,G\in C_0^\infty(\Z_{\ge0}\times\Zsstar)$.  Now pick any $G_0$ so that $G_0(m,x)\neq G_0(m+n,x)$ for all $m$ and all $x$, note that this is always possible.  Thus we get
\begin{equation*}
\tilde{\delta}(F) = \frac{\tilde{\delta}(G_0)}{G_0-\beta^n(G_0)}(F-\beta^n (F)):=H(F-\beta^n(F))
\end{equation*}
where $H$ is some complex-valued function defined on $\Z_{\ge0}\times\Zsstar$.  Applying $\delta$ to the relation $VM_FV^*=M_{\alpha F}$ gives
\begin{equation*}
\delta(M_{\alpha F}) = \delta(V)M_FV^*+V\delta(M_F)V^* + VM_F\delta(V^*)\,.
\end{equation*}
Using this and the definition of $\tilde{\delta}$ yields
\begin{equation*}
-V^nM_{\alpha\beta G}M_{\alpha F} + V^nM_{\alpha H(\alpha F-\alpha\beta^n F)} + V^nM_{\beta^{n-1}F}M_G = V^nM_{H(\alpha F-\beta^n\alpha F)}\,.
\end{equation*}
This implies that
\begin{equation*}
(-\alpha\beta(G) +\alpha(H)-H)\alpha(F) + (G+H)\beta^{n-1}(F)-\alpha(H)\alpha\beta^n(F)=0\,.
\end{equation*}
Since $\alpha(H)\alpha\beta^n(F)=\alpha(H)\beta^{n-1}(F)$ for all $x$, we get
\begin{equation*}
(-G+\alpha(H) - H)\alpha(F)=(\alpha(H)-G-H)\alpha\beta^{n-1}(F)
\end{equation*}
for every $F$.
Choose $F$ such that $\alpha(F)$ and $\alpha\beta^{n-1}(F)$ are nonzero, then it follows that 
\begin{equation}\label{GHequation}
G=\alpha(H)-H.
\end{equation}
Explicitly, equation \eqref{GHequation} is the following recurrence equation:
\begin{equation*}
\left\{
\begin{aligned}
&G(k,x)=H(k-1,x) - H(k,x) &&\textrm{ if }k\ge1\\
&G(0,x) = -H(0,x) &&\textrm{ if }k=0\,.
\end{aligned}\right.
\end{equation*}
This system has a formal solution, namely:
\begin{equation*}
H(k,x) = -\sum_{j=0}^kG(j,x) = \sum_{j=k+1}^\infty G(j,x) - \sum_{j=0}^\infty G(j,x):=R + r\,.
\end{equation*}
Since $\delta(V^*)\in I_s^\infty$, we have that $G\in C_0^\infty(\Z_{\ge0}\times\Zsstar)$.  Consequently the above sums converge absolutely and are RD.  Thus $R\in C_0^\infty(\Z_{\ge0}\times\Zsstar)$ and $r\in C(\Z_s^*)$. 

Thus, so far we  have
\begin{equation*}
\delta(V)=-V^{n+1}M_{\beta(G)}\,,\quad\delta(V^*)=V^{n-1}M_{G}\,,\quad\textrm{and}\quad\delta(M_F)=V^nM_{H(F-\beta^n F)}
\end{equation*}
where $G=\alpha(H)-H$.  We wish to show $\delta$ has the following form:
\begin{equation*}
\delta(a) = [V^nM_{H},a]
\end{equation*}
for $n>0$.   A direct calculation verifies the above formula works for $a=V$, $a=V^*$, and $a=M_F$.  
Thus, by continuity, the formula must be true for all $a\in HS^\infty(s)$.
 A similar proof works for when $n<0$.
\end{proof}
Notice that in the above statement a continuous $n$-covariant derivation $\delta:HS^\infty(s)\to I_s^\infty$ is inner if and only if the corresponding $r\in C(\Zsstar)$ is a constant.

 Recall the space
\begin{equation*}
C^\infty(\R/\Z\times\Zsstar) = \left\{\Lambda(\theta,x)=\sum_{n\in\Z}e^{2\pi in\theta}\lambda_n(x) : \lambda_n \in C(\Zsstar) \text{ and } \{\underset{x}{\textrm{sup }}|\lambda_n(x)|\}_n\textrm{ is RD}\right\}\,.
\end{equation*}
Given $\Lambda\in C^\infty(\R/\Z\times\Zsstar)$ recall that we can associate a ``generalized'' Toeplitz operator
\begin{equation*}
\mathcal{T}(\Lambda) = \sum_{n\ge0}V^nM_{\lambda_n} + \sum_{n<0}M_{\lambda_n}(V^*)^{-n}\,.
\end{equation*}

The next proposition classifies all continuous derivations on $HS^\infty(s)$ whose range is in $I_s^\infty$.  Proposition \ref{fancytoeplitz} implies that $[\mathcal{T}(\Lambda),a]$ is in $I_s^\infty$ for any $a\in HS^\infty(s)$.

\begin{prop}\label{d_decomp}
If $\delta:HS^\infty(s)\to I_s^\infty$ is a continuous derivation, then there exists an element $\Lambda\in C^\infty(\R/\Z\times\Zsstar)$ and an element $b\in I_s^\infty$ such that
\begin{equation*}
\delta(a) = [\mathcal{T}(\Lambda) + b, a]\,.
\end{equation*}
Moreover, $\delta$ is inner if and only if $\Lambda$ does not depend on $x$.
\end{prop}

\begin{proof}
First write $\delta$ as its Fourier decomposition: 
\begin{equation*}
\delta(a)=\sum_{n\in\Z} \delta_n(a)
\end{equation*}
By Proposition \ref{norm_conv}, this sum is convergent with respect to the $M,N$-norms.  Further, Proposition \ref{ncovariantdeltan} implies that $\delta_n$ is $n$-covariant. 
By Proposition \ref{covariant_inner}, $\delta_n$ may be represented by a special commutator.  We only consider the case $n\ge0$ as the case $n<0$ is similar.

For $n\geq 0$, we get 
\begin{equation*}
\delta_n(a)=[V^n(M_{r_n}+M_{R_n}),a]
\end{equation*}
where $r_n\in C(\Zsstar)$  and  $R_n\in C^\infty_0(\Z_{\ge0}\times\Zsstar)$ for each $n$.  

Recall from the proof of Proposition \ref{covariant_inner},
\begin{equation*}\delta_n(V^*)=V^{n-1}M_{G_n}\end{equation*}
where $G_n\in C^\infty_0(\Z_{\ge0} \times \Zsstar)$ and has the following relation with $R_n$ and $r_n$:
\begin{equation*}
\sum_{j=k+1}^\infty G_n(j,x) - \sum_{j=0}^\infty G_n(j,x)=R_n(k,x) + r_n(x)\,.
\end{equation*}
The goal is to establish rapid decay properties of $R_n$ and $r_n$.

Since $\delta(V^*)\in I_s^\infty$, we have that for all $N$ and $j$ there exists a constant such that 
\begin{equation*}
\|\dpp^j(\delta(V^*))(1+\PP)^N\| \leq const(N,j).
\end{equation*}
However by the definition of $\delta_n$ we get
\begin{equation*}
\delta_n(V^*) = \int_0^1 e^{-2\pi i n\theta} \rho_\theta\delta\rho_\theta^{-1}(V^*)\,d\theta = \int_0^1 e^{-2\pi i(n+1)\theta}\rho_\theta \delta(V^*)\,d\theta.
\end{equation*}
Note that  
\begin{equation*}
\begin{aligned}
\|\dpp^j(\delta_n(V^*))(1+\PP)^N\| &= \left\|\int_0^1 e^{-2\pi i(n+1)\theta}\rho_\theta\dpp^j(\delta(V^*))(1+\PP)^N\,d\theta\right\|\leq const(N,j).
\end{aligned}\end{equation*}
On the other hand, computing the norm explicitly in the above inequality we have
\begin{equation*}
\begin{aligned} 
\|\dpp^j(V^{n-1}M_{G_n})(1+\PP)^N\|=(n-1)^j \sup_{m,x}|G_n(m,x)|(1+m)^N\leq const(N,j).
\end{aligned}\end{equation*}
So $\left\{\sup_x |G_n(m,x)|\right\}_{m,n}$ is RD.  It follows that $\{\sup_x |R_n(m,x)|\}_{m,n}$ is an RD sequence and $\{\sup_x|r_n(x)|\}_n$ is an RD sequence.  A similar argument shows the same RD results for $n<0$.

Define $\Lambda$ by 
\begin{equation*}
\Lambda(\theta,x) := \sum_{n\in \Z} e^{2\pi in\theta}r_n(x).
\end{equation*}
The above argument shows that $\Lambda\in C^\infty(\R/\Z\times \Zsstar)$.  Also define $b$ by
\begin{equation*}
b:=\sum_{n\ge0} V^n M_{R_n} + \sum_{n<0}M_{R_n}(V^*)^{-n}.
\end{equation*}
The above argument also shows that $b\in I_s^\infty$.  

We now claim that
\begin{equation*}
\delta(a)=[\mathcal{T}(\Lambda)+b,a]
\end{equation*}
for all $a\in HS^\infty(s)$. Indeed, since $\Lambda\in C^\infty(\R/\Z\times\Zsstar)$ and $b\in I_s^\infty$ both sides of the equation are continuous derivations $HS^\infty(s)\to I_s^\infty$ with the same Fourier components and thus must be equal.

Notice again that $\mathcal{T}(\Lambda)$ is a Toeplitz operator if and only if $\Lambda$ does not depend on $x$, thus $\delta$ is inner if and only if $\Lambda$ does not depend on $x$.

%
\end{proof}

\subsection{Classification of Derivations}
For $\Lambda\in C^\infty(\R/\Z\times\Zsstar)$ define the following map
\begin{equation}\label{defnfrakd}
a\mapsto \mathfrak{d}_\Lambda(a) := [\mathcal{T}(\Lambda),a]
\end{equation}
and notice by Proposition \ref{fancytoeplitz} that $\mathfrak{d}_\Lambda$ is a continuous derivation from $HS^\infty(s)$ to $I_s^\infty$, appearing in Proposition \ref{d_decomp}.
\begin{prop}\label{frakdprop}
Let $\Lambda\in C^\infty(\R/\Z\times \Zsstar)$ and let $\mathfrak{d}_\Lambda:HS^\infty(s)\to I_s^\infty$ be as above.  Then we have the following properties.
\begin{enumerate}
\item $\mathfrak{d}_\Lambda$ is inner if and only if $\Lambda$ does not depend on $x$.
\item $\mathfrak{d}_\Lambda=0$ if and only if $\Lambda$ does not depend on $\theta$.
\end{enumerate}
\end{prop}
\begin{proof}
The first item has already been established, thus we only need to prove the second.  Consider $\mathfrak{d}_\Lambda$ acting on $V$:
\begin{equation*}\begin{aligned}
0 = \mathfrak{d}_\Lambda(V) &= \left[\sum_{n\ge0} V^n M_{\lambda_n} + \sum_{n< 0} M_{\lambda_n}(V^*)^{-n},V\right] \\
&= \left[\sum_{n\ge0} V^nM_{\lambda_n},V\right] + \left[\sum_{n< 0} M_{\lambda_n}(V^*)^{-n},V\right]
\end{aligned}
\end{equation*}
and notice that the first commutator is 0.  The second commutator is the sum of projections $P_{<-n}$ with coefficients $M_{\lambda_n}$.  Since the sum must be 0 and the projections $P_{<-n}$ are not zero, $M_{\lambda_n}$ must be $0$ for $n<0$.  

Similarly, applying $\mathfrak{d}_\Lambda$ to $V^*$ we have $M_{\lambda_n}=0$ for $n>0$.  So $\Lambda(\theta,x)=\Lambda(0,x)$ for all $x$, and thus $\Lambda$ does not depend on $\theta$.
\end{proof}

To obtain uniqueness in the classification of derivations we need to make further choices of $\Lambda$'s.
One convenient way to do it is to consider the following expectations in $C^\infty(\R/\Z\times \Zsstar)$:  $E_1:C^\infty(\R/\Z\times \Zsstar)\to C^\infty(\R/\Z)$ and $E_2:C^\infty(\R/\Z\times \Zsstar)\to C(\Zsstar)$ given by 
\begin{equation}
E_1(\Lambda) = \int_{\Zsstar} \Lambda(\theta,x)\,dx \text{\quad and\quad } E_2(\Lambda)=\int_0^1 \Lambda(\theta,x)\,d\theta.
\end{equation}

\begin{prop}\label{bothe1e2zero}
Let $\Lambda\in C^\infty(\R/\Z\times \Zsstar)$.  Then we have the following:
\begin{enumerate}
\item There is a $\Lambda'\in C^\infty(\R/\Z\times \Zsstar)$ so that $E_2(\Lambda')=0$ and $\mathfrak{d}_\Lambda=\mathfrak{d}_{\Lambda'}$.
\item If $E_2(\Lambda)=0$ then there is a $\Lambda'\in C^\infty(\R/\Z\times \Zsstar)$ so that $E_1(\Lambda')=E_2(\Lambda')=0$ and $\mathfrak{d}_\Lambda=\mathfrak{d}_{\Lambda'}+\tilde{\delta}$ where  $\tilde{\delta}$ is an inner derivation on $HS^\infty(s)$.
\item If $E_1(\Lambda)=E_2(\Lambda)=0$ then $\mathfrak{d}_\Lambda$ is inner if and only if $\Lambda=0$.
\end{enumerate}
\end{prop}
\begin{proof}
For the first item, note that $E_2(\Lambda)$ does not depend on $\theta$.  Define $\Lambda' = \Lambda-E_2(\Lambda)$; then \begin{equation*}
\mathfrak{d}_\Lambda=\mathfrak{d}_{\Lambda-E_2(\Lambda)}.
\end{equation*} 
Moreover, since $E_2$ is an expectation we have 
\begin{equation*}
E_2(\Lambda')=E_2(\Lambda-E_2(\Lambda))=0.
\end{equation*}

For the second item, put $\Lambda'=\Lambda-E_1(\Lambda)$.  Since $E_1$ and $E_2$ are expectations, we have
\begin{equation*}
E_1(\Lambda') = E_1(\Lambda - E_1(\Lambda)) = 0
\end{equation*}
 and 
\begin{equation*}
E_2(\Lambda')=E_2(\Lambda)-E_2(E_1(\Lambda)) = 0-E_1(E_2(\Lambda))=0
\end{equation*}
where Fubini's Theorem is used to interchange the order of the expectations.

Thus we have that
\begin{equation*}
\mathfrak{d}_{\Lambda'} = \mathfrak{d}_\Lambda -\mathfrak{d}_{E_1(\Lambda)}.
\end{equation*}  
Note that $\mathfrak{d}_{E_1(\Lambda)}$ does not depend on $x$ and hence is inner, proving part 2.  Therefore we may always choose $\Lambda$ such that $E_1(\Lambda)=E_2(\Lambda)=0$.

Now suppose $E_1(\Lambda)=E_2(\Lambda)=0$ and $\mathfrak{d}_\Lambda$ is inner.  Since $\mathfrak{d}_\Lambda$ is inner, $\Lambda$ does not depend on $x$.  Thus we have that
\begin{equation*}
0 = E_1(\Lambda) = \Lambda.
\end{equation*}
So $\Lambda=0$.  It is clear that if $\Lambda=0$ then $\mathfrak{d}_\Lambda$ is inner.
\end{proof}
%

The following theorem describes a decomposition of any continuous derivation on $HS^\infty(s)$ into special derivations and is the main result of this section.

\begin{theo}\label{der_theo}
If $\delta:HS^\infty(s)\to HS^\infty(s)$ is a continuous derivation, then there exist unique  $\varphi\in C^\infty(\R/\Z)$ and $\Lambda\in C^\infty(\R/\Z \times\Zsstar)$ with $E_1(\Lambda)=E_2(\Lambda)=0$ such that
\begin{equation*}
\delta = \delta_\varphi + \mathfrak{d}_\Lambda +  \tilde{\delta}\,,
\end{equation*}
where $\delta_\varphi$ was defined in Proposition \ref{special_der_delta_F}, $\mathfrak{d}_\Lambda$ defined in equation \eqref{defnfrakd}, 
and $\tilde{\delta}$ is an inner derivation on $HS^\infty(s)$.
\end{theo}

\begin{proof}
Since $\delta$ is a continuous derivation, equation \eqref{classofdelta} defines a $\varphi\in C^\infty(\R/\Z)$.  Consider the following calculation:
\begin{equation*}
q(\delta(a)-\delta_\varphi(a)) = \frac{1}{2\pi i} \varphi(\theta)\frac{d}{d\theta}q(a) - \frac{1}{2\pi i} \varphi(\theta)\frac{d}{d\theta}q(a) = 0
\end{equation*}
for all $a\in HS^\infty(s)$.  Thus the difference $\delta-\delta_\varphi$ is a continuous derivation from $HS^\infty(s)$ to $I_s^\infty$. Observe that the range of $\delta_\varphi$ is contained in $I_s^\infty$ if and only if $\varphi=0$.  This gives a unique $\varphi$ in this decomposition.  By Proposition \ref{d_decomp} there exists a $\Psi\in C^\infty(\R/\Z\times \Zsstar)$ and $b\in I_s^\infty$ such that 
\begin{equation*}
(\delta-\delta_\varphi)(a) = \mathfrak{d}_\Psi(a) + [b,a]
\end{equation*}
for any $a\in HS^\infty(s)$.  By Proposition \ref{bothe1e2zero} choose $\Psi$ so that $E_2(\Psi)=0$.  Define $\Lambda$ by $\Lambda = \Psi-E_1(\Psi)$ and again by Proposition \ref{bothe1e2zero}, we have that $E_1(\Lambda)=E_2(\Lambda)=0$.  Moreover, we have 
\begin{equation*}
\delta(a) = \delta_\varphi(a) + \mathfrak{d}_\Lambda(a) + [T(E_2(\Psi))+b,a]
\end{equation*}
for every $a\in HS^\infty(s)$.  The uniqueness of $\Lambda$ follows from Proposition \ref{bothe1e2zero}. This completes the proof.  

\end{proof}

\section{K-Theory}
In this section, we compute the $K$-Theory of $HS(s)$. In particular, we investigate the six term exact sequence in $K$-Theory induced by the short exact sequence 
\begin{equation*}
0 \to I_s \to  HS(s) \to C(\R / \Z) \to 0.
\end{equation*}
These results are summarized in the following proposition. 
\begin{prop}
The $K$-Theory of $HS(s)$ is given by 
\begin{equation*}
K_0(HS(s)) \cong C(\Zsstar , \Z) \quad \text{and} \quad K_1(HS(s)) \cong 0. 
\end{equation*}
\end{prop}

\begin{proof}
The short exact sequence 
\begin{equation*}
0 \to I_s \to HS(s) \to C(\R / \Z) \to 0
\end{equation*}
induces the following six term exact sequence in $K$-Theory:
\begin{equation*}
\begin{tikzcd}
 K_0(I_s) \arrow{r}   & K_0(HS(s))  \arrow{r} & K_0(C(\R / \Z)) \arrow{d}{\textup{ exp }}    \\
 K_1(C(\R / \Z)) \arrow{u}{\textup{ ind }} & K_1(HS(s)) \arrow{l} & K_1(I_s) \arrow{l}
\end{tikzcd}
\end{equation*}
Since $I_s \cong C(\Zsstar) \otimes \mathcal{K}$ by Theorem \ref{ideal isomorphism}, stability of $K$-Theory shows that 
\begin{equation*}
K_0(I_s) \cong K_0(C(\Zsstar)) \quad \text{and} \quad K_1(I_s) \cong K_1(C(\Zsstar)). 
\end{equation*}
Denote by $C(\Zsstar,\Z)$ the group of continuous functions on $\Zsstar$ with values in $\Z$. Since $\Zsstar$ is totally disconnected, we have that
\begin{equation*}
K_0(I_s) \cong C(\Zsstar,\Z)
\end{equation*}
by Exercise 3.4 in \cite{RLL}. Moreover, from Example III.2.5 in \cite{KD}, $C(\Zsstar)$ is an approximately finite dimensional algebra, and hence 
\begin{equation*}
K_1(I_s) \cong 0. 
\end{equation*}
Since $K_1(C(\R / \Z)) \cong \Z$ is generated by the class of the generating unitary $x \mapsto e^{2 \pi i x}$, which lifts to the isometry $V \in HS(s)$, it follows from Proposition 9.2.4 in \cite{RLL} that 
\begin{equation*}
\textup{ind}([x \mapsto e^{2 \pi i x}]) = -[I - VV^*]_0. 
\end{equation*}
In particular, the index map is injective. By exactness at $K_1(C(\R / \Z))$, we see that the map $K_1(HS(s)) \to K_1(C(\R/\Z))$ is trivial. On the other hand $K_1(I_s) \cong 0$, and so by exactness at $K_1(HS(s))$ we see that the map  $K_1(HS(s))\to K_1(C(\R/\Z))$ is injective. It follows that 
\begin{equation*}
K_1(HS(s)) \cong 0.
\end{equation*}
Rewriting the six term exact sequence above with the computed terms, we obtain the following exact sequence
\begin{equation*}
0 \to \Z \to C(\Z_s^\times, \Z) \to K_0(HS(s)) \to \Z \to 0,
\end{equation*}
where the $\Z$ on the right side of the sequence corresponds to $K_0(C(\R / \Z)),$ and is generated by the class of the identity. The range of the map $\Z \to C(\Zsstar , \Z)$ is generated by $-[I - VV^*]_0$. which is seen by direct calculation to correspond to the constant function $-1$ in $C(\Zsstar, \Z)$. Taking the quotient, we obtain a short exact sequence
\begin{equation*}
0 \to C(\Zsstar,\Z)/\Z \to K_0(HS(s)) \to \Z \to 0. 
\end{equation*}
The map $\Z \to K_0(HS(s))$ given by $[I]_0 \to [I_0]$ is a right inverse for the map $K_0(HS(s)) \to \Z$. This is easily seen since the surjective homomorphism $HS(s) \to C(\R/\Z)$ is unital. It follows by the splitting lemma that 
\begin{equation*}
K_0(HS(s)) \cong C(\Zsstar,\Z)/\Z \oplus \Z. 
\end{equation*}
However, the short exact sequence of abelian groups 
\begin{equation*}
0 \to \Z \to C(\Zsstar, \Z) \to C(\Zsstar,\Z)/\Z \to 0,
\end{equation*}
wherein the map $\Z \to C(\Zsstar, \Z)$ is given by $k \mapsto (x\mapsto -k)$, is easily seen to be left split via the evaluation homomorphism $C(\Zsstar , \Z) \to \Z$ given by $f \mapsto -f(1)$. From this and the splitting lemma we conclude that 
\begin{equation*}
K_0(HS(s)) \cong C(\Zsstar,\Z)/\Z \oplus \Z \cong C(\Zsstar , \Z). 
\end{equation*}
This completes the proof. 
\end{proof}

We close the section with a comment on $K$-Homology of $HS(s)$. Since $C(\Zsstar , \Z)$ is free (see exercise 7.7.5 in \cite{HR}), using the Universal Coefficient Theorem \cite{RS UCT} we obtain 
$$K^1(HS(s)) \cong 0 \quad \text{and} \quad K^0(HS(s)) \cong \Hom(C(\Zsstar , \Z),\Z).$$

\section{Appendix}
The main objective of this appendix is to review some generalities about
the concept of crossed products by endomorphisms under special hypotheses first considered in \cite{Mu}.

Let $A$ be a unital C$^*$-algebra and $\alpha : A \to A$ be a monomorphism with hereditary range. 
Clearly $\alpha(1)$ is a projection and notice that $\textrm{Ran}(\alpha)$, the range of $\alpha$, is contained in $\alpha(1)A\alpha(1)$. It easily follows, see \cite{E} Proposition 4.1, that if $\textrm{Ran}(\alpha)$ is a hereditary subalgebra of $A$ then we have 
$$\textrm{Ran}(\alpha)=\alpha(1)A\alpha(1).$$ 
Thus we can define a left inverse $\beta$ of $\alpha$ by:
\begin{equation}\label{beta_def_ref}
\beta(a):=\alpha^{-1}(\alpha(1)a\alpha(1)).
\end{equation}
It is easy to see that we have the following properties:
\begin{equation*}
\beta\alpha(a)=a,\, \alpha\beta(a)=\alpha(1)a\alpha(1), \textrm{ and } \beta(1)=1.
\end{equation*}
Moreover,  the map $\beta: A\to A$ is linear, continuous, positive and has the ``transfer" property:
\begin{equation*}
\beta(\alpha(a)b)=a\beta(b),
\end{equation*}
and thus is an example of Exel's transfer operator \cite{E}. Because we have $\alpha\beta(a)=\alpha(1)a\alpha(1)$ it is a complete transfer operator of \cite{BL} and in particular a non-degenerate transfer operator of \cite{E}, i.e. it satisfies $\alpha(\beta(1))=\alpha(1)$. Also notice that if the algebra $A$ is commutative, then $\beta$ is also an endomorphism.

Stacey in \cite{St} (and Murphy in \cite{Mu}) defines the crossed product $A\rtimes_\alpha\N$ as the universal unital C$^*$-algebra with generators $M_a$ for $a\in A$ such that the map $a\mapsto M_a$ is a $*$-homomorphism
, and an isometry $V$,   subject to the relation:
\begin{equation*}
M_{\alpha(a)} = VM_aV^*, \ a \in A.
\end{equation*}

This implies that $M_{\alpha(1)}=VV^*$. We also have
\begin{equation*}
V^*M_aV = M_{\beta(a)},
\end{equation*}
which follows from the following calculation:
\begin{equation*}
V^*M_aV = V^*VV^*M_aVV^*V = V^*M_{\alpha(1)}M_{a}M_{\alpha(1)}V = V^*M_{\alpha\beta(a)}V = V^*VM_{\beta(a)}V^*V = M_{\beta(a)}.
\end{equation*}

If the algebra $A$ is commutative we get a stronger relation:
\begin{equation}\label{strong_rel}
M_aV =  VM_{\beta(a)}.
\end{equation}
Indeed, we can manipulate as follows:
\begin{equation*}
VM_{\beta(a)}=VV^*M_aV=M_{\alpha(1)}M_aV=M_aM_{\alpha(1)}V=M_aVV^*V=M_aV.
\end{equation*}

Similarly to crossed products by automorphisms the map $a \mapsto M_a$ is injective by Corollary 1.3 in \cite{boyd1993faithful}.  For the remainder of the appendix we identify $a$ with $M_a$ and drop the $M_a$ notation.

Given a unital C$^*$-algebra $A$, an endomorphism $\alpha : A \to A$, and a transfer operator $\beta$, the Exel’s crossed product is the universal C$^*$-algebra generated by a copy of $A$ and element $V$ subject to the relations:
\begin{enumerate}
\item $Va=\alpha(a)V$
\item $V^*aV=\beta(a)$
\item (Redundancy Condition) If for $a\in\overline{A\alpha(A)A}$ and $k\in\overline{AVV^*A}$ we have $abV=kbV$ for all $b\in A$ then $a=k$.
\end{enumerate}
Notice that, compared to Stacey's definition, $V$ does not have to be an isometry, unless $\beta(1)=1$. If $\beta(1)=1$ and so $V$ is an isometry then the first condition of Exel's crossed product is weaker than Stacey's condition $VaV^*=\alpha(a)$. On the other hand, there is no analog of condition 2 above in Stacey's scheme, where in general $V^*AV$ is not contained in $A$. Notice also that the less intuitive redundancy condition is empty if the two ideals $\overline{A\alpha(A)A}$ and $\overline{AVV^*A}$ are equal. This is because if $abV=0$ for all $b\in A$ then $abVV^*c=0$ for all $b,c\in A$, so $ax=0$ for all $x\in\overline{AVV^*A}$ and in particular for $x=a^*$, which implies $a=0$.

Interestingly, if $\alpha$ is a monomorphism with hereditary range then both seemingly quite different concepts of crossed products by endomorphisms coincide if we take $\beta$ to be the transfer operator given by equation \eqref{beta_def_ref}.  Then, as we have seen above, Stacey's conditions imply the first two of Exel's conditions. Also, since $VV^*=\alpha(1)$, we have
\begin{equation*}
\overline{AVV^*A}=\overline{A\alpha(1)A}\subseteq \overline{A\alpha(A)A}=\overline{A\alpha(1)A\alpha(1)A}\subseteq
\overline{A\alpha(1)A}.
\end{equation*}
Consequently, the two ideals $\overline{A\alpha(A)A}$ and $\overline{AVV^*A}$ are equal and the redundancy condition is trivially satisfied.

Conversely, starting with Exel's conditions for a monomorphism $\alpha$ with hereditary range and the transfer operator $\beta$ given by equation \eqref{beta_def_ref} we see that $\beta(1)=1$ and so $V$ is an isometry. The other Stacey condition follows from the redundancy condition since $\alpha(a)\in\overline{A\alpha(A)A}$, $VaV^*=aVV^*\in\overline{AVV^*A}$ and we have:
\begin{equation*}
VaV^*bV=Va\beta(b)=\alpha(a)\alpha\beta(b)V=\alpha(a)\alpha(1)b\alpha(1)V=\alpha(a)bV,
\end{equation*}
where we used $\alpha(1)V=V\cdot 1=V$.

To understand the structure of the crossed product C$^*$-algebra $A\rtimes_\alpha\N$,  the following Fourier type series  expansions are frequently useful. 
Using the crossed product commutation relations one can arrange the polynomials in $V$, $V^*$ and $a$ so that $a$'s are in front of powers of $V$ and powers of $V^*$ are in front of $a$'s. This leads to the following observation:
the vector space  $\mathcal{F}(A,V)$ consisting of finite sums
\begin{equation}\label{Fourier1}
x=\sum_{n\ge0}a_n V^n + \sum_{n<0}(V^*)^{-n} a_n\,,
\end{equation}
where $a_n\in A$, is a dense $*$-subalgebra of $A\rtimes_\alpha\N$.
Additionally, using $V^n = V^n(V^*)^nV^n$ and its adjoint
one can always choose coefficients $a_n$ such that
\begin{equation*}
a_n=a_nV^n(V^*)^n ,\ n\geq 0 \textrm{ and } a_n=V^{-n}(V^*)^{-n}a_n,\ n<0.
\end{equation*}
Moreover such coefficients are unique. This follows from the existence of an expectation in $A\rtimes_\alpha\N$.
It is defined using a one-parameter group of automorphisms given on generators by formulas $\rhot(V)=e^{2\pi i\theta}V$ and $\rhot(a)=a$ and extending by universality to the whole crossed product $A\rtimes_\alpha\N$. 
Define the map $E:A\rtimes_\alpha\N\to A\rtimes_\alpha\N$ by
\begin{equation*}
E(a)=\int_0^1 \rhot(a)\,d\theta \,.
\end{equation*}
Clearly, $E$ is an expectation in $A\rtimes_\alpha\N$ onto $A\subseteq A\rtimes_\alpha\N$.
For any $x\in \mathcal{F}(A,V)$ given by a finite sum as in equation \eqref{Fourier1}
we compute directly to get
\begin{equation*}
a_n = \left\{
\begin{aligned}
&E(x(V^*)^{n}) &&\textrm{ if }n\ge0\\
&E(V^{-n}x) &&\textrm{ if }n<0\,.
\end{aligned}\right.
\end{equation*}

Alternatively, if the algebra $A$ is commutative and we have the stronger condition, equation \eqref{strong_rel}, then we can arrange the polynomials in $V$, $V^*$ and $a$ in the opposite order to  equation \eqref{Fourier1}. Then the $*$-subalgebra $\mathcal{F}(A,V)$ can be described as consisting of finite sums
\begin{equation}\label{Fourier2}
x=\sum_{n\ge0}V^n b_n + \sum_{n<0}b_n (V^*)^{-n}\,
\end{equation}
with $b_n\in A$.  The advantage of equation \eqref{Fourier2} over equation \eqref{Fourier1} is that the coefficients $b_n$ are uniquely determined by $x$ and are given by the formulas:
\begin{equation*}
b_n = \left\{
\begin{aligned}
&E((V^*)^{n}x) &&\textrm{ if }n\ge0\\
&E(xV^{-n}) &&\textrm{ if }n<0\,.
\end{aligned}\right.
\end{equation*}
Notice that by the properties of expectations we have the following norm estimates for coefficients in   equations \eqref{Fourier1} and \eqref{Fourier2}:
\begin{equation*}
||a_n||\leq ||x||  \textrm{ and } ||b_n||\leq ||x||
\end{equation*}
for all $n\in\Z$.

Suppose $\pi: A\rtimes_\alpha\N\to B(H)$ is a representation of $A\rtimes_\alpha\N$ in a Hilbert space $H$. We say that $\pi$ satisfies the O'Donovan conditions if
\begin{enumerate}
\item $\pi$ restricted to $A$ is a faithful representation of the C$^*$-algebra $A$.
\item for every $x\in\mathcal{F}(A,V)$ we have:
\begin{equation*}
||\pi(a_0)||= ||\pi(b_0)||\leq ||\pi(x)||.
\end{equation*}
\end{enumerate}
A key result from \cite{boyd1993faithful} says that if a representation $\pi: A\rtimes_\alpha\N\to B(H)$ satisfies the O'Donovan conditions then it is a faithful representation of the crossed product C$^*$-algebra $A\rtimes_\alpha\N$.


\begin{thebibliography}{99}


\bibitem{BL}
Bakhtin, V.I., Lebedev, A.V., When a C$^*$-algebra is a Coefficient Algebra for a Given Endomorphism, arXiv:1503.02897 [math.AP].


\bibitem{BC}
Blackadar, B., and Cuntz J., Differential Banach Algebra Norms and Smooth Subalgebras of C$^*$-algebras, {\it J. Oper. Theory}, 26(2):255-282, 1991.

\bibitem{Bo} 
Bost, J.B., Principe d'Oka, K-Theorie et Systemes Dynamiques non Commutatifs, {\it Invent. Math.}, 101(2):261-334, 1990.

\bibitem{BEJ}
Bratteli, O., Elliott, G. A., and Jorgensen, P. E. T., Decomposition of Unbounded Derivations into Invariant and Approximately Inner Parts, {\it Jour. Reine Ang. Math.}, 346: 166-193, 1984.

\bibitem{boyd1993faithful}
Boyd, S., Keswani, N., and Raeburn, I., Faithful Representations of Crossed Products by Endomorphisms, \textit{Proc. AMS}, 118(2):427-436, 1993.

\bibitem{Connes}
Connes, A., {\it Non-Commutative Differential Geometry}, Academic Press, 1994.

\bibitem{CuntzVershik}
Cuntz, J., and Vershik, A. C$^*$-algebreas associated with endomorphims and polymorphisms of compact abelian groups, \textit{Comm. Math. Phys.}, 321(1):157-179, 2013. 

\bibitem{CuLi}
Cuntz, J and Li, X., C$^*$-algebras associated with integral domains and crossed prod-ucts by actions on adele spaces, \textit{J. Noncommut. Geom.}, 5 , 1–37, 2011.

\bibitem{KD}
Davidson, K., {\it C$^*$-algebras by Example}, American Mathematical Society, 1996.



\bibitem{E}
Exel, R., A New Look at the Crossed Product of a C$^*$-algebra by an Endomorphism, {\it Erg. Theo. Dyn. Sys.}, 23(6):1733–1750, 2003.


\bibitem{HKMP} 
Hebert, S., Klimek, S., McBride, M., and Peoples, J.W., Crossed Product C$^*$-algebras Associated with $p$-adic Multiplication, 	arXiv:2312.09381.

\bibitem{H}
Hadfield, T., The Noncommutative Geometry of the Discrete Heisenberg Group, {\it Houston J. Math.}, 29(2):453-481, 2002.

\bibitem{HR}
Higson, N. and Roe, J., {\it Analytic K-Homology}, Oxford University Press, 2001.

\bibitem{Hirshberg}
Hirshberg, I. On C$^*$-algebras associated to certain endomorphisms of discrete groups,
\textit{New York J. Math.} 8:99–109, 2002.


\bibitem{K}
Katznelson, Y., {\it An Introduction to Harmonic Analysis}, Cambridge University Press, 2004.




\bibitem{KMP2}
Klimek, S., McBride, M., and Peoples, J.W., Aspects of Noncommutative Geometry of Bunce-Deddens Algebras,  {\it Jour. Noncomm. Geom.}, 17(4):1391–1423, 2023.

\bibitem{KMP3}
Klimek, S., McBride, M., and Peoples, J.W., Noncommutative Geometry of the Quantum Disk, {\it Annal. Func. Anal.}, 53(13):1-55, 2022.

\bibitem{KMP4}
Klimek, S., McBride, M., and Peoples, J.W., A Note on Quantum Odometers, {\it Sci. China Math.} 66:1555–1568, 2023.

\bibitem{Laca}
Laca, M., From endomorphisms to automorphisms and back: dilations and full corners,
\textit{J. London Math. Soc.} 61(3):893–904, 2000.

\bibitem{monothetic}
Klimek, S. and McBride, M., Unbounded Derivations in Algebras Associated with Monothetic Groups, \textit{Jour. Aust. Math. Soc.}, 111(3):345-371, 2021.







%
%


\bibitem{LarsenLi}
Larsen, N.S. and Li, X., The 2-adic ring C$^*$-algebra of the integers and its representations,
\textit{J. Funct. Anal.} 262(4):1392–1426, 2012.


\bibitem{Me} Meyer, R., {\it Local and Analytic Cyclic Homology}, European Mathematical Society (EMS), 2007.


\bibitem{Mu} Murphy, G. J., Crossed Products of C*-algebras by Endomorphisms, {\it Int. Eq. Oper. Theo.} 24, 298-319, 1996.

\bibitem{newns} Newns, W. F., and Walker, A. G., Tangent Planes to a Differentiable Manifold, {\it J. London Math. Soc.}, 31:400-407, 1956.

%




\bibitem{RLL}
Rordam, M., Larsen, F., and Lausten, N.J. {\it An Introduction to $K$-Theory for C$^*$-algebras}, Cambridge University Press, 2000.

\bibitem{RS UCT}
Rosenberg, J. and Schochet, C., The Kunneth Theorem and the Universal Coefficient Theorem for Kasporov's Generalised $K$-functor, {\it Duke Math. J.}, 55:431-474, 1987.

\bibitem{S}
Sakai, S., {\it Operator Algebras in Dynamical Systems}, Cambridge University Press, 1991.

\bibitem{St} Stacey, P.J., Crossed Products of C$^*$-algebras by $*$-endomorphisms, {\it J. Austral. Math. Soc. Ser. A}, 54:204-212, 1993.


\end{thebibliography}
\end{document}